\documentclass[12pt]{amsart}
\usepackage{amssymb,amsmath,tabularx,mathrsfs,mathbbol,yfonts,upgreek}
\usepackage{amsthm,verbatim,comment}
\usepackage{marginnote}
\usepackage[bookmarks=true]{hyperref}
\usepackage{pstricks,pst-node,pst-plot}
\usepackage{geometry}
\usepackage{stmaryrd}
\usepackage{paralist}
\usepackage[all]{xy}
\usepackage{array}
\usepackage{url}
\usepackage{longtable}
\numberwithin{equation}{section}

\DeclareSymbolFontAlphabet{\mathbb}{AMSb}
\DeclareSymbolFontAlphabet{\mathbbl}{bbold}

\newtheorem{thm}{Theorem}[section]
\newtheorem{thmx}{Theorem}

\newtheorem{lem}[thm]{Lemma}
\newtheorem{prop}[thm]{Proposition}
\newtheorem{cor}[thm]{Corollary}

\theoremstyle{definition}
\newtheorem{defn}[thm]{Definition}
\newtheorem{nota}[thm]{Notation}
\newtheorem{eg}[thm]{Example}

\theoremstyle{remarks}

\newtheorem*{rem*}{Remarks}

\newtheoremstyle{case}{}{}{}{}{}{:}{ }{}
\theoremstyle{case}

\newcommand{\B}{\mathcal{B}}

\newcommand{\F}{\mathbb{F}}

\title[On Terwilliger $\F$-algebras of quasi-thin association schemes]{On Terwilliger $\F$-algebras of quasi-thin association schemes}
\begin{document}
\author{Yu Jiang}
\address[Y. Jiang]{Division of Mathematical Sciences, Nanyang Technological University, SPMS-MAS-05-34, 21 Nanyang Link, Singapore 637371.}
\email[Y. Jiang]{jian0089@e.ntu.edu.sg}


\begin{abstract} In \cite{Han1}, Hanaki defined the Terwilliger algebras of association schemes over a commutative unital ring. In this paper, we call the Terwilliger algebras of association schemes over a field $\F$ the Terwilliger $\F$-algebras of association schemes and study the Terwilliger $\F$-algebras of quasi-thin association schemes. As main results, we determine the $\F$-dimensions, the semisimplicity, the Jacobson radicals, and the algebraic structures of the Terwilliger $\F$-algebras of quasi-thin association schemes. We also get some results with independent interests.
\vspace{-1em}
\end{abstract}
\maketitle
\noindent{\textbf{Keywords.} Association scheme; Terwilliger $\F$-algebra; Quasi-thin scheme\\
\textbf{Mathematics Subject Classification 2020.} 05E30 (primary), 05E16 (secondary)}
\section{Introduction}
Let $X$ denote a nonempty finite set and call an association scheme on $X$ a scheme.

The subconstituent algebras of schemes, introduced by Terwilliger in \cite{T}, are now known as the Terwilliger algebras of schemes. Recall that a Terwilliger algebra of a scheme is a finite-dimensional semisimple complex algebra that contains the complex adjacency algebra of this scheme. So it contains the combinatorial information of this scheme and can control the structure of this scheme (see \cite{T} and \cite{Rie}). However, it is difficult to determine the algebraic structure of a Terwilliger algebra of a scheme.

In \cite{Han1}, Hanaki introduced the Terwilliger algebras of schemes over a commutative unital ring. Moreover, he also called the Terwilliger algebras of schemes over a field of positive characteristic the modular Terwilliger algebras of schemes and studied their basic properties. In particular, he proved that the algebraic structure of a modular Terwilliger algebra of a scheme depends on the characteristic of the coefficient field (see Theorem \ref{T;Maincitation}). So determining the algebraic structure of a modular Terwilliger algebra of a scheme is also difficult.

Fix a field $\F$ of characteristic $p$ and call the Terwilliger algebras of schemes over $\F$ the Terwilliger $\F$-algebras of schemes. In this paper, we determine the $\F$-dimensions, the semisimplicity, the Jacobson radicals, the algebraic structures of the Terwilliger $\F$-algebras of quasi-thin schemes. Fix a scheme $S=\{R_0, R_1,\ldots, R_d\}$. Notice that the notion of a Terwilliger $\F$-algebra of $S$ generalizes the notions of a Terwilliger algebra of $S$ and a modular Terwilliger algebra of $S$. So we can describe the algebraic structures of the Terwilliger algebras and modular Terwilliger algebras of quasi-thin schemes by our main results. Fix $x\in X$ and let $\mathcal{T}$ denote the Terwilliger $\F$-algebra of $S$ with respect to $x$. Let $\mathcal{J}$ denote the Jacobson radical of $\mathcal{T}$. Given $R_a, R_b\in S$, let $R_aR_b$ be the complex product of $R_a, R_b$ and $R_{a'}$ be the transpose relation of $R_a$. Write $k_a$ for the valency of $R_a$. We state our main results as follows.
\begin{thmx}\label{T;A}
Assume that $S=\{R_0, R_1,\ldots, R_d\}$ is a quasi-thin scheme. Then the $\F$-dimension of $\mathcal{T}$ is the sum of $(d+1)^2$, $|\{(a,b): R_a, R_b\in S,\ |R_{a'}R_b|=2\}|$, and $|\{(c,e): R_c, R_e\in S,\ (c,e)\ \text{is a bad pair of $S$}\}|$.
\end{thmx}
For the definition of a bad pair of $S$, one can refer to Definition \ref{D;badpair}. Observe that the definition of a bad pair of $S$ only depends on the intersection numbers of $S$.
\begin{thmx}\label{T;B}
Assume that $S$ is a quasi-thin scheme. Then $\mathcal{T}$ is semisimple if and only if $p\neq2$ or $p=2$ and $S$ is a thin scheme.
\end{thmx}
\begin{thmx}\label{T;C}
Assume that $p=2$ and $S$ is a quasi-thin scheme. Then $\mathcal{J}=\mathcal{J}_1$.
\end{thmx}
For the definition of $\mathcal{J}_1$, one can refer to Notations \ref{N;notation3} and \ref{N;notation1}. According to Theorems \ref{T;B} and \ref{T;C}, we can determine the Jacobson radicals of Terwilliger $\F$-algebras of quasi-thin schemes.
\begin{thmx}\label{T;D}
Assume that $S$ is a quasi-thin scheme. Every Terwilliger $\F$-algebra of $S$ is isomorphic to $\mathcal{T}$ as $\F$-algebras.
\end{thmx}
\begin{thmx}\label{T;E}
Assume that $S=\{R_0, R_1,\ldots, R_d\}$ is a quasi-thin scheme and let $M_n(\F)$ be the full matrix algebra of $\F$-square matrices of size $n$.
\begin{enumerate}[(i)]
\item [\em (i)] Assume further that $S$ is a thin scheme. Then $\mathcal{T}\cong M_{|X|}(\F)$ as $\F$-algebras.
\item [\em (ii)] Assume further that $S$ is not a thin scheme. Set $\mathcal{A}_i=\{a: R_a\in S,\ k_a=i\}$ for any $i\in\{1,2\}$. Let $\sim$ be a binary relation on $\mathcal{A}_2$, where, for any $b, c\in\mathcal{A}_2$, $b\sim c$ if and only if $(b,c)\in\{(g,h): |R_{g'}R_h|=2\ \text{or}\ (g,h)\ \text{is a bad pair of $S$}\}$. Then $\sim$ is an equivalence relation on $\mathcal{A}_2$.
Let $\mathcal{C}_0, \mathcal{C}_1,\ldots, \mathcal{C}_{\gamma}$ be exactly all pairwise distinct equivalence classes of $\mathcal{A}_2$ with respect to $\sim$.
    \begin{enumerate}[(a)]
    \item [\em (a)] If $p\neq 2$, there exist minimal two-sided ideals $\mathcal{I}_{-1}, \mathcal{I}_0,\ldots, \mathcal{I}_{\gamma}$ of $\mathcal{T}$ such that the following assertions hold.
        \begin{enumerate}[(1)]
        \item [\em (a1)] For any $j\in\{-1,0,\ldots,\gamma\}$, $\mathcal{I}_j$ is a unital $\F$-algebra.
        \item [\em (a2)] The ideal $\mathcal{I}_{-1}$ is a simple $\F$-algebra with $\F$-dimension $(d+1)^2$.
        \item [\em (a3)] As $\F$-algebras, $\mathcal{T}=\mathcal{I}_{-1}\oplus \bigoplus_{k=0}^{\gamma}\mathcal{I}_k\cong \mathcal{I}_{-1}\oplus \bigoplus_{k=0}^{\gamma}M_{|\mathcal{C}_k|}(\F)$.
         \end{enumerate}
        \item [\em (b)] If $p=2$, let $\mathcal{T}/\mathcal{J}$ be the quotient $\F$-algebra of $\mathcal{T}$ with respect to $\mathcal{J}$. Then there exist minimal two-sided ideals $\mathcal{K}_{-1}, \mathcal{K}_0,\ldots, \mathcal{K}_{\gamma}$ of $\mathcal{T}/\mathcal{J}$ such that the following assertions hold.
            \begin{enumerate}[(1)]
            \item [\em (b1)] For any $u\in \{-1, 0,\ldots, \gamma\}$, $\mathcal{K}_u$ is a unital $\F$-algebra.
            \item [\em (b2)] As $\F$-algebras, $\mathcal{T}/\mathcal{J}=\mathcal{K}_{-1}\oplus\bigoplus_{v=0}^{\gamma}\mathcal{K}_v\cong M_{|\mathcal{A}_1|}(\F)\oplus \bigoplus_{v=0}^{\gamma}M_{|\mathcal{C}_v|}(\F)$.
            \end{enumerate}
    \end{enumerate}
\end{enumerate}
\end{thmx}
This paper is organized as follows. In Section $2$, we gather the basic notation and preliminaries. In Section $3$, we study the triple products of Terwilliger $\F$-algebras of schemes and present some needed results. In Section $4$, we show Theorem \ref{T;A}. Theorems \ref{T;B} and \ref{T;C} are also proved in Sections $5$ and $6$ respectively. In Section $7$, we finish the proofs of Theorems \ref{T;D} and \ref{T;E}.
\subsection*{Acknowledgement}
The presented paper is motivated by \cite[Problem 4.3]{Han} and \cite{Han1}. The author gratefully thanks Prof. Akihide Hanaki for letting him know many interesting problems on the modular representations of association schemes.
\section{Basic notation and preliminaries}
In this section, we set up the basic notation and present some preliminary results. We assume that the reader is familiar with the theory of association schemes. For a general background on this topic, one may refer to \cite{EI}, \cite{Z2}, or \cite{Z3}.
\subsection{General conventions}
Throughout this paper, fix a field $\F$ of characteristic $p$ and a nonempty finite set $X$. Let $\mathbb{N}$ denote the set of all natural numbers and set $\mathbb{N}_0=\mathbb{N}\cup\{0\}$. For any $a\in \mathbb{N}_0$, define $[a]=\{b\in \mathbb{N}_0: b\leq a\}$ and set $[-a]=\varnothing$ if $a\neq 0$. For a nonempty subset $Y$ of an $\F$-linear space, let $\langle Y\rangle_\F$ be the $\F$-linear space generated by $Y$. If $Y=\{y\}$, set $\langle y\rangle_\F=\langle \{y\}\rangle_\F$ and let $\langle\varnothing\rangle_\F$ be the zero space. The addition, the multiplication, the scalar multiplication of matrices displayed in this paper are the usual matrix operations. A scheme is an association scheme on $X$.
\subsection{Theory of schemes}
Let $S=\{R_0, R_1,\ldots, R_d\}$ be a partition of the cartesian product $X\times X$, where $R_a\neq \varnothing$ for all $a\in [d]$. Then $S$ is called a scheme of class $d$ if the following conditions hold:
\begin{enumerate}[(i)]
\item $R_0=\{(b,b): b\in X\}$;
\item For any $c\in [d]$, there exists $c'\in [d]$ such that $\{(e,f): (f,e)\in R_c\}=R_{c'}\in S$;
\item For any $i,j,k\in [d]$ and $(m,n), (\hat{m},\hat{n})\in R_k$, we have the following equality: $|\{\ell\in X: (m,\ell)\in R_i,\ (\ell,n)\in R_j\}|=|\{\hat{\ell}\in X: (\hat{m},\hat{\ell})\in R_i,\ (\hat{\ell},\hat{n})\in R_j\}|$.
\end{enumerate}

Throughout this paper, $S=\{R_0, R_1,\ldots, R_d\}$ denotes a fixed scheme of class $d$. Every member of $X$ is called a vertex of $S$. Every member of $S$ is called a relation of $S$. According to the definition of $S$, for any $u,v,w\in [d]$ and $(g,h)\in R_w$, we define $$p_{uv}^w=|\{r\in X: (g,r)\in R_u,\ (r,h)\in R_v\}|\in \mathbb{N}_0.$$ The number $p_{uv}^w$ is called the intersection number of $S$ with respect to $R_u$, $R_v$, $R_w$. An intersection number of $S$ always means an intersection number of $S$ with respect to three given relations of $S$. For any $R_q\in S$, let $k_q$ denote $p_{qq'}^0$ and call $k_q$ the valency of $R_q$. For any $s\in X$, we set $sR_q=\{\hat{s}\in X: (s, \hat{s})\in R_q\}$ and observe that $|sR_{q}|=k_q$. As $s$ is chosen from $X$ arbitrarily and $R_q\neq \varnothing$, notice that $k_q>0$.
For any nonempty subsets $U, V$ of $S$, we set $$UV=\{R_z\in S: \exists\ R_x\in U,\ \exists\ R_y\in V,\ p_{xy}^z>0\}.$$ The operation between $U$ and $V$ is called the complex multiplication of $U$ and $V$.
For any $R_{\alpha}, R_{\beta}\in S$, the complex product of $R_{\alpha}$ and $R_{\beta}$ is defined to be $\{R_{\alpha}\}\{R_{\beta}\}$ and is denoted by $R_{\alpha}R_{\beta}$. We collect some results of the intersection numbers of $S$ as follows. These results shall be repeatedly used in the following sections.
\begin{lem}\label{L;Intersectionnumber}
Assume that $R_a, R_b, R_c\in S$.
\begin{enumerate}[(i)]
\item [\em (i)] \cite[Lemmas 1.1.1 (ii) and 1.1.2 (iii)]{Z3} We have $p_{a'b'}^{c'}=p_{ba}^{c}$ and $k_a=k_{a'}$.
\item [\em (ii)] \cite[Lemma 1.1.3 (iii)]{Z3} We have $\sum_{e=0}^d p_{ae}^b=k_a$.
\item [\em (iii)]\cite[Lemma 1.1.3 (iv)]{Z3} We have $k_ak_b=\sum_{e=0}^dp_{ab}^e k_e$.
\item [\em (iv)]\cite[Proposition 2.2 (vi)]{EI} We have $k_cp_{ab}^c=k_ap_{cb'}^a=k_bp_{a'c}^b$.
\item [\em (v)]\cite[Lemma 1.5.2]{Z3} The number $|R_aR_b|$ is no more than the greatest common divisor of $k_a$ and $k_b$.
\end{enumerate}
\end{lem}
We are interested in some special schemes. Let us state their definitions as follows.

The scheme $S$ is called a thin scheme if the valency of every relation of $S$ is one.

The scheme $S$ is called a quasi-thin scheme if the valency of every relation of $S$ is at most two. So every thin scheme is a quasi-thin scheme. For the details of some properties of quasi-thin schemes, one may refer to \cite{H2}, \cite{H3}, \cite{J}, \cite{MP}, and \cite{MX}.

The scheme $S$ is called a $p'$-valenced scheme if the valency of every relation of $S$ is not divisible by $p$. So every thin scheme is a $p'$-valenced scheme. Observe that a quasi-thin scheme is a $p'$-valenced scheme if $p\neq2$. Notice that every scheme is a $0'$-valenced scheme.
We close this subsection by a property of quasi-thin schemes.
\begin{lem}\label{L;Nonzerointersection}
Let $n$ be the number of nonzero intersection numbers of $S$. If $S$ is a quasi-thin scheme, then   $n=(d+1)^2+|\{(a,b): R_a, R_b\in S,\ |R_{a'}R_b|=2\}|$.
\end{lem}
\begin{proof}
Since $S$ is a quasi-thin scheme, notice that $k_c\leq 2$ for all $R_c\in S$. By Lemma \ref{L;Intersectionnumber} (v) and the definition of complex product of two relations of $S$, we have
\begin{align*}
n&=|\{(e,f): R_e, R_f\in S,\ |R_eR_f|=1\}|+2|\{(g,h): R_g, R_h\in S,\ |R_gR_h|=2\}|\\
&=|\{(i,j): R_i, R_j\in S\}|+|\{(g,h): R_g, R_h\in S,\ |R_gR_h|=2\}|\\
&=(d+1)^2+|\{(g,h): R_g, R_h\in S,\ |R_gR_h|=2\}|.
\end{align*}
Therefore it is enough to prove that $$|\{(a,b): R_a, R_b\in S,\ |R_{a'}R_b|=2\}|=|\{(g,h): R_g, R_h\in S,\ |R_gR_h|=2\}|.$$ For convenience, we set
\begin{align*}
\mathcal{G}=\{(a,b): R_a, R_b\in S,\ |R_{a'}R_b|=2\}\ \text{and}\
\mathcal{H}=\{(g,h): R_g, R_h\in S,\ |R_gR_h|=2\}.
\end{align*}
We may assume further that $\mathcal{G}\neq \varnothing$. For any $(\ell,m)\in \mathcal{G}$, notice that $(\ell', m)\in \mathcal{H}$. Let $\phi$ denote the map from $\mathcal{G}$ to $\mathcal{H}$ that sends every $(r,s)$ to $(r',s)$. It is clear that $\phi$ is a bijective map from $\mathcal{G}$ to $\mathcal{H}$. In particular, we have $|\mathcal{G}|=|\mathcal{H}|$. We are done.
\end{proof}
\subsection{Matrix algebras}\label{Sc:subsection}
Let $M_{X}(\F)$ be the full matrix algebra of $\F$-square matrices whose rows and columns are labeled by the members of $X$. Let $M=(a_{ij})\in M_X(\F)$. For any $b,c\in X$, the $b$-row (resp. $b$-column) of $M$ means the row (resp. column) of $M$ labeled by $b$. The entry $a_{bc}$ lies in the $b$-row and the $c$-column of $M$. It is called the $(b,c)$-entry of $M$. The $b$-row sum (resp. $b$-column sum) of $M$ means $\sum_{e\in X}a_{be}$ (resp. $\sum_{e\in X}a_{eb}$). Let $M^t$ be the transpose of $M$ and $L^t=\{N^t: N\in L\}$ for any $L\subseteq M_X(\F)$. Let $I$, $J$, $O$ denote the identity matrix, the all-one matrix, the all-zero matrix of $M_X(\F)$ respectively. For any $u,v\in X$, let $E_{uv}$ denote the $(0,1)$-matrix of $M_X(\F)$ whose unique nonzero entry is the $(u,v)$-entry. So $E_{bc}E_{uv}=\delta_{cu}E_{bv}$, where $\delta_{\alpha\beta}$ is the Kronecker delta of vertices $\alpha,\beta$ whose values are in $\F$. It is obvious that $\{E_{yz}: y,z\in X\}$ is an $\F$-basis of $M_X(\F)$. We end this subsection with a lemma.
\begin{lem}\label{L;Radical}
Let $A$ be a unital $\F$-subalgebra of $M_X(\F)$ and $\mathfrak{J}$ be the Jacobson radical of $A$. If $A^t=A$, then $\mathfrak{J}^t=\mathfrak{J}$.
\end{lem}
\begin{proof}
Since $\mathfrak{J}^t\subseteq A^t=A$ and $\mathfrak{J}$ is the Jacobson radical of $A$, observe that $\mathfrak{J}^t$ is a two-sided nilpotent ideal of $A$, which implies that $\mathfrak{J}^t\subseteq \mathfrak{J}$. For any $M\in \mathfrak{J}$, notice that $M=(M^t)^t\in \mathfrak{J}^t$ as $M^t\in \mathfrak{J}^t\subseteq \mathfrak{J}$. So $\mathfrak{J}^t=\mathfrak{J}$. The desired lemma is proved.
\end{proof}
\subsection{Terwilliger $\F$-algebras of schemes}
Let $\mathbb{Z}$ be the integer ring and $\F_p$ denote the prime subfield of $\F$. For any $a\in \mathbb{Z}$, let $\overline{a}$ be the image of $a$ under the surjective unital ring homomorphism from $\mathbb{Z}$ to $\F_p$.

Let $R_b\in S$. The adjacency $\F$-matrix with respect to $R_b$ is denoted by $A_b$. It is defined to be the $(0,1)$-matrix $(a_{ce})$ of $M_{X}(\F)$, where we have $a_{ce}=1$ if and only if $(c,e)\in R_b$. According to the definition of $S$, we observe that $A_0=I$, $A_b^t=A_{b'}$, and $\sum_{f=0}^dA_f=J$. Moreover, for any $R_g, R_h\in S$, we have
\begin{equation*}\label{Eq:preliminary1}
\tag{2.1} A_gA_h=\sum_{\ell=0}^d\overline{p_{gh}^{\ell}}A_{\ell}.
\end{equation*}
Let $y\in X$ and $R_z\in S$. The dual $\F$-idempotent with respect to $y$, $R_z$ is denoted by $E_z^*(y)$.  It is defined to be the sum $\sum_{i\in yR_z}E_{ii}$. Let $R_j\in S$ and define $\delta_{\alpha\beta}$ to be the Kronecker delta of integers $\alpha, \beta$ whose values are in $\F$. By the definition of $E_z^*(y)$, note that $E_z^*(y)E_j^*(y)=\delta_{zj}E_z^*(y)$, $\sum_{k=0}^dE_k^*(y)=I$, and $JE_z^*(y)J=\overline{k_z}J$. Moreover, for any $M=(b_{mn})\in M_X(\F)$, write $M=\sum_{m\in X}\sum_{n\in X}b_{mn}E_{mn}$ and notice that
\begin{equation*}\label{Eq:preliminary2}
\tag{2.2} E_z^*(y)ME_j^*(y)=\sum_{q\in yR_z}\sum_{r\in yR_j}b_{qr}E_{qr}.
\end{equation*}

The Terwilliger $\F$-algebra of $S$ with respect to $y$, denoted by $\mathcal{T}(y)$, is defined to be the $\F$-subalgebra of $M_X(\F)$ generated by $A_0, A_1,\ldots ,A_d$ and $E_0^*(y), E_1^*(y), \ldots, E_d^*(y)$. Call $E_z^*(y)A_bE_j^*(y)$ a triple product of $\mathcal{T}(y)$ and notice that it is a $(0,1)$-matrix by \eqref{Eq:preliminary2}. A Terwilliger $\F$-algebra of $S$ means the Terwilliger $\F$-algebra of $S$ with respect to a vertex of $S$. We now fix $x\in X$. Throughout this paper, let $\mathcal{T}=\mathcal{T}(x)$ and $E_s^*=E_s^*(x)$ for all $R_s\in S$. By the definition of $\mathcal{T}$, we have $\mathcal{T}^t=\mathcal{T}$. Let $\mathcal{J}$ be the Jacobson radical of $\mathcal{T}$. Let $w\in \mathbb{N}_0$ and $R_{i_u}, R_{j_u}, R_{\ell_u}\in S$ for all $u\in [w]$. We set
\[\prod_{v=0}^w(E_{i_v}^*A_{j_v}E_{\ell_v}^*)=E_{i_0}^*A_{j_0}E_{\ell_0}^*E_{i_1}^*A_{j_1}E_{\ell_1}^*\cdots E_{i_w}^*A_{j_w}E_{\ell_w}^*\in \mathcal{T}.\]
For example, if $w=2$, then $\prod_{v=0}^2(E_{i_v}^*A_{j_v}E_{\ell_v}^*)=E_{i_0}^*A_{j_0}E_{\ell_0}^*E_{i_1}^*A_{j_1}E_{\ell_1}^*E_{i_2}^*A_{j_2}E_{\ell_2}^*.$
\begin{lem}\label{L;generalresults}
The following assertions hold.
\begin{enumerate}[(i)]
\item [\em (i)] \cite[Lemma 3.2]{Han1} Let $R_a, R_b, R_c\in S$. Then the $e$-row of $E_a^*A_bE_c^*$ has exactly $p_{cb'}^a$ ones for any $e\in xR_a$. Moreover, the $f$-column of $E_a^*A_bE_c^*$ has exactly $p_{ab}^c$ ones for any $f\in xR_c$. In particular, $E_a^*A_bE_c^*\neq O$ if and only if $p_{ab}^c\neq 0$.
\item [\em (ii)] The set $\{E_i^*A_jE_\ell^*: R_i, R_j, R_\ell\in S,\ p_{ij}^{\ell}\neq0\}$ is an $\F$-linearly independent subset of $\mathcal{T}$. In particular, the number of nonzero intersection numbers of $S$ is equal to $|\{E_i^*A_jE_\ell^*: R_i, R_j, R_\ell\in S,\ p_{ij}^{\ell}\neq0\}|$.
\item [\em (iii)] The set $\{E_i^*JE_\ell^*: R_i, R_\ell\in S\}$ is an $\F$-linearly independent subset of $\mathcal{T}$ whose size equals $(d+1)^2$.
\end{enumerate}
\end{lem}
\begin{proof}
By (i), notice that $O\notin \{E_i^*A_jE_\ell^*: R_i, R_j, R_\ell\in S,\ p_{ij}^{\ell}\neq0\}$. As $\sum_{u=0}^dA_u=J$ and $\{xR_0, xR_1,\ldots, xR_d\}$ is a partition of $X$, the first assertion of (ii) is shown by \eqref{Eq:preliminary2}. According to the first assertion of (ii), the second assertion of (ii) is proved. Since $\{xR_0, xR_1,\ldots, xR_d\}$ is a partition of $X$, (iii) is also shown by \eqref{Eq:preliminary2}.
\end{proof}
\section{Triple products of Terwilliger $\F$-algebras of schemes}
In this section, we provide some results of the triple products of $\mathcal{T}$. These results shall be utilized in the following sections. We first present some results about the multiplication of two triple products of $\mathcal{T}$.
\begin{lem}\label{L;Case1}
Let $R_i, R_j, R_\ell, R_m, R_n\in S$. Assume that $E_{i}^*A_jE_\ell^*\neq O\neq E_\ell^*A_mE_{n}^*$. If $\min\{k_i, k_j\}=1$ or $\min\{k_m, k_n\}=1$, then we have $$E_{i}^*A_jE_{\ell}^*A_mE_{n}^*=E_{i}^*A_jA_mE_{n}^*=\sum_{k=0}^d\overline{p_{jm}^k}E_{i}^*A_kE_{n}^*.$$
\end{lem}
\begin{proof}
Since we have $E_{i}^*A_jE_\ell^*\neq O\neq E_{\ell}^*A_mE_{n}^*$, notice that $p_{ij}^\ell>0$ and $p_{\ell m}^n>0$ by Lemma \ref{L;generalresults} (i). If $\min\{k_i,k_j\}=1$, then we have $k_i=1$ or $k_j=1$. By Lemma \ref{L;Intersectionnumber} (v), for any $R_a\in S$, we deduce that $p_{ij}^a>0$ only if $a=\ell$. In particular, by Lemma \ref{L;generalresults} (i) again, we get that $E_{i}^*A_jE_a^*\neq O$ only if $a=\ell$. We thus have
\begin{align*}\label{Eq:1} \tag{3.1} E_{i}^*A_jE_{\ell}^*A_mE_{n}^*=E_{i}^*A_j(\sum_{k=0}^dE_{k}^*)A_mE_{n}^*=E_{i}^*A_jA_mE_{n}^*=\sum_{k=0}^d\overline{p_{jm}^k}E_{i}^*A_kE_{n}^*
\end{align*}
by Lemma \ref{L;generalresults} (i) and \eqref{Eq:preliminary1}. If $\min\{k_m,k_n\}=1$,
by Lemma \ref{L;Intersectionnumber} (i), we also have $\min\{k_n, k_{m'}\}=1$. Since we have known that $p_{\ell m}^n\neq 0$, notice that $E_n^*A_{m'}E_\ell^*\neq O$ by Lemmas \ref{L;Intersectionnumber} (iv) and \ref{L;generalresults} (i). Following the analysis of the case $\min\{k_i,k_j\}=1$, we thus have
\begin{align*}\label{Eq:2} \tag{3.2}E_{n}^*A_{m'}E_\ell^*A_{j'}E_{i}^*=E_{n}^*A_{m'}A_{j'}E_{i}^*=\sum_{k=0}^d\overline{p_{m'j'}^{k'}}E_{n}^*A_{k'}E_{i}^*=\sum_{k=0}^d\overline{p_{jm}^{k}}E_{n}^*A_{k'}E_{i}^*
\end{align*}
by \eqref{Eq:preliminary1} and Lemma \ref{L;Intersectionnumber} (i). We can also obtain \eqref{Eq:1} by taking transposes in \eqref{Eq:2}. The desired lemma thus follows.
\end{proof}
\begin{lem}\label{L;Complexproductone}
Let $R_i, R_j, R_\ell, R_m\in S$. If $R_{i'}R_\ell=\{R_j\}$, then $E_i^*A_mE_\ell^*\neq O$ if and only if $m=j$. Moreover, $E_i^*A_jE_\ell^*=E_i^*JE_\ell^*$. In particular, if $k_i=1$ or $k_\ell=1$, then there is a unique $R_n\in S$ such that $E_i^*A_nE_\ell^*\neq O$. Moreover, $E_i^*A_nE_\ell^*=E_i^*JE_\ell^*$.
\end{lem}
\begin{proof}
As $R_{i'}R_\ell=\{R_j\}$, notice that $p_{i'\ell}^j>0$, which implies that $p_{ij}^\ell>0$ by Lemma \ref{L;Intersectionnumber} (iv). So $E_i^*A_jE_\ell^*\neq O$ by Lemma \ref{L;generalresults} (i). Conversely, if $E_i^*A_mE_\ell^*\neq O$, we have $p_{im}^\ell>0$ by Lemma \ref{L;generalresults} (i). We thus get $p_{i'\ell}^m>0$ by Lemma \ref{L;Intersectionnumber} (iv). So $R_m\in R_{i'}R_\ell$, which implies that $m=j$. We thus have $O\neq E_i^*A_jE_\ell^*=\sum_{k=0}^dE_i^*A_kE_\ell^*=E_i^*JE_\ell^*$.

If $k_i=1$ or $k_\ell=1$, then $|R_{i'}R_\ell|=1$ by Lemma \ref{L;Intersectionnumber} (i) and (v). So the remaining two assertions are shown by the first assertion and the second assertion.
\end{proof}
\begin{lem}\label{L;Case2}
Let $R_i, R_j, R_\ell, R_m, R_n\in S$. If $E_{i}^*A_jE_\ell^*\neq O\neq E_\ell^*A_mE_{n}^*$ and $k_\ell=1$, then we have
$$E_{i}^*A_jE_{\ell}^*A_mE_{n}^*=E_i^*JE_n^*=\sum_{k=0}^dE_{i}^*A_kE_{n}^*.$$
\end{lem}
\begin{proof}
Since $k_\ell=1$ and $E_{i}^*A_jE_\ell^*\neq O\neq E_\ell^*A_mE_{n}^*$, we have $E_i^*A_jE_\ell^*=E_i^*JE_\ell^*$ and $E_\ell^*A_mE_n^*=E_\ell^*JE_n^*$ by Lemma \ref{L;Complexproductone}. So the desired lemma follows as $JE_\ell^*J=J$.
\end{proof}
\begin{lem}\label{L;Case3}
Let $R_i, R_j, R_\ell, R_m, R_n\in S$.
\begin{enumerate}[(i)]
\item [\em (i)] If $k_ik_j=p_{ij}^\ell k_\ell$, then we have
$$E_{i}^*A_jE_{\ell}^*A_mE_{n}^*=E_{i}^*A_jA_mE_{n}^*=\sum_{k=0}^d\overline{p_{jm}^k}E_{i}^*A_kE_{n}^*.$$
\item [\em (ii)] If $k_\ell k_m=p_{\ell m}^nk_n$ and $k_\ell=k_n$, then we have
$$E_{i}^*A_jE_{\ell}^*A_mE_{n}^*=E_{i}^*A_jA_mE_{n}^*=\sum_{k=0}^d\overline{p_{jm}^k}E_{i}^*A_kE_{n}^*.$$
\end{enumerate}
\end{lem}
\begin{proof}
Since $k_ik_j=p_{ij}^\ell k_\ell$, for any $R_a\in S$, Lemma \ref{L;Intersectionnumber} (iii) tells us that $p_{ij}^a>0$ only if $a=\ell$. By Lemma \ref{L;generalresults} (i), we thus have $E_{i}^*A_jE_{a}^*\neq O$ only if $a=\ell$. According to Lemma \ref{L;generalresults} (i) and \eqref{Eq:preliminary1}, (i) is proved as we can deduce that
\begin{align*}\label{Eq:3}
\tag{3.3}E_{i}^*A_jE_{\ell}^*A_mE_{n}^*=E_{i}^*A_j(\sum_{k=0}^dE_{k}^*)A_mE_{n}^*=E_{i}^*A_jA_mE_{n}^*=\sum_{k=0}^d\overline{p_{jm}^k}E_{i}^*A_kE_{n}^*.
\end{align*}

For (ii), since $k_\ell k_m=p_{\ell m}^nk_n$ and $k_\ell=k_n$, Lemma \ref{L;Intersectionnumber} (i) and (iv) tell us that
\begin{align*}\label{Eq:4}
\tag{3.4}k_nk_{m'}=k_\ell k_m=p_{\ell m}^nk_n=p_{nm'}^\ell k_\ell.
\end{align*}
For any $R_a\in S$, by \eqref{Eq:4} and Lemma \ref{L;Intersectionnumber} (iii), observe that $p_{nm'}^a>0$ only if $a=\ell$. By Lemma \ref{L;generalresults} (i), we thus have $E_{n}^*A_{m'}E_{a}^*\neq O$ only if $a=\ell$. Following the analysis of (i), we deduce that
\begin{align*}\label{Eq:5}
\tag{3.5}E_{n}^*A_{m'}E_{\ell}^*A_{j'}E_{i}^*=E_{n}^*A_{m'}A_{j'}E_{i}^*=\sum_{k=0}^d\overline{p_{m'j'}^{k'}}E_{n}^*A_{k'}E_{i}^*=\sum_{k=0}^d\overline{p_{jm}^{k}}E_{n}^*A_{k'}E_{i}^*
\end{align*}
by \eqref{Eq:preliminary1} and Lemma \ref{L;Intersectionnumber} (i). We also get \eqref{Eq:3} by taking transposes in \eqref{Eq:5}.
\end{proof}
The following lemmas focus on the multiplication of many triple products of $\mathcal{T}$.
\begin{lem}\label{L;Case4}
Let $a\in \mathbb{N}_{0}$ and assume that $R_{i_b}, R_{j_b}, R_{\ell_b}\in S$ for all $b\in [a]$. If we have $k_{i_b}=k_{\ell_b}$ for all $b\in [a]$, then the $r$-row sum of $\prod_{z=0}^a (E_{i_z}^*A_{j_z}E_{\ell_z}^*)$ equals the $c$-column sum of $\prod_{z=0}^a (E_{i_z}^*A_{j_z}E_{\ell_z}^*)$ for any $r\in xR_{i_0}$ and $c\in xR_{\ell_a}$.
\end{lem}
\begin{proof}
We work by induction. We first check the base case $a=0$. Since $k_{i_0}=k_{\ell_0}$ and $E_{i_0}^*A_{j_0}E_{\ell_0}^*$ is a $(0,1)$-matrix, by Lemmas \ref{L;Intersectionnumber} (iv) and \ref{L;generalresults} (i), the $r$-row sum of $E_{i_0}^*A_{j_0}E_{\ell_0}^*$ equals the $c$-column sum of $E_{i_0}^*A_{j_0}E_{\ell_0}^*$ for any $r\in xR_{i_0}$ and $c\in xR_{\ell_0}$. 

Assume that $a>0$. As $k_{i_b}=k_{\ell_b}$ for all $b\in [a]$, we assume that the $e$-row sum of $\prod_{z=0}^{m} (E_{i_z}^*A_{j_z}E_{\ell_z}^*)$ equals the $f$-column sum of $\prod_{z=0}^{m} (E_{i_z}^*A_{j_z}E_{\ell_z}^*)$ for any $m\in [a-1]$, $e\in xR_{i_0}$, and $f\in xR_{\ell_m}$. If $\prod_{z=0}^{a}(E_{i_z}^*A_{j_z}E_{\ell_z}^*)=O$, then the desired lemma holds. We thus can assume further that $\prod_{z=0}^{a}(E_{i_z}^*A_{j_z}E_{\ell_z}^*)\neq O$, which implies that $\ell_{a-1}=i_a$. Therefore we can set
\begin{align*}\label{Eq:6}
\tag{3.6}xR_{\ell_{a-1}}=xR_{i_a}=\{y_1, y_2,\ldots, y_{k_{\ell_{a-1}}}\}.
\end{align*}
For any $r\in xR_{i_0}$, let $g$ denote the $r$-row sum of $\prod_{z=0}^{a-1}(E_{i_z}^*A_{j_z}E_{\ell_z}^*)$. For any $c\in xR_{\ell_a}$, write $h$ for the $c$-column sum of $E_{i_a}^*A_{j_a}E_{\ell_a}^*$.

We list some facts to get the $r$-row sum and the $c$-column sum of $\prod_{z=0}^{a}(E_{i_z}^*A_{j_z}E_{\ell_z}^*)$.
\begin{enumerate}[(i)]
\item By \eqref{Eq:preliminary2} and \eqref{Eq:6}, all nonzero entries of $\prod_{z=0}^{a-1}(E_{i_z}^*A_{j_z}E_{\ell_z}^*)$ are in the columns labeled by $y_1, y_2,\ldots, y_{k_{\ell_{a-1}}}$.
\item By \eqref{Eq:preliminary2} and \eqref{Eq:6}, all nonzero entries of $E_{i_a}^*A_{j_a}E_{\ell_a}^*$ are in the rows labeled by $y_1,y_2,\ldots, y_{k_{\ell_{a-1}}}$.
\item As $k_{i_a}=k_{\ell_a}$ and $E_{i_a}^*A_{j_a}E_{\ell_a}^*$ is a $(0,1)$-matrix, by Lemmas \ref{L;Intersectionnumber} (iv), \ref{L;generalresults} (i), and  \eqref{Eq:6}, the $y_n$-row sum of $E_{i_a}^*A_{j_a}E_{\ell_a}^*$ equals $h$ for all $n\in [k_{\ell_{a-1}}]\setminus\{0\}$.
\item By (i) and (ii), all nonzero entries of $(\prod_{z=0}^{a-1}(E_{i_z}^*A_{j_z}E_{\ell_z}^*))^t$ are in the rows labeled by $y_1, y_2,\ldots, y_{k_{\ell_{a-1}}}$. All nonzero entries of $E_{\ell_a}^*A_{j_a'}E_{i_a}^*$ are in the columns labeled by $y_1,y_2,\ldots, y_{k_{\ell_{a-1}}}$.
\item By the inductive hypothesis and \eqref{Eq:6}, the $y_n$-row sum of $(\prod_{z=0}^{a-1}(E_{i_z}^*A_{j_z}E_{\ell_z}^*))^t$ equals $g$ for all $n\in [k_{\ell_{a-1}}]\setminus\{0\}$.
\end{enumerate}
As $\prod_{z=0}^{a}(E_{i_z}^*A_{j_z}E_{\ell_z}^*)\!\!=\!(\prod_{z=0}^{a-1}(E_{i_z}^*A_{j_z}E_{\ell_z}^*))(E_{i_a}^*A_{j_a}E_{\ell_a}^*)$,
by (i), (ii), (iii), the $r$-row sum of $\prod_{z=0}^{a}(E_{i_z}^*A_{j_z}E_{\ell_z}^*)$ is $gh$. As $(\prod_{z=0}^{a}(E_{i_z}^*A_{j_z}E_{\ell_z}^*))^t=(E_{\ell_a}^*A_{j_a'}E_{i_a}^*)(\prod_{z=0}^{a-1}(E_{i_z}^*A_{j_z}E_{\ell_z}^*))^t$,
by (iv) and (v), the $c$-row sum of $(\prod_{z=0}^{a}(E_{i_z}^*A_{j_z}E_{\ell_z}^*))^t$ is $gh$. So the $c$-column sum of $\prod_{z=0}^{a}(E_{i_z}^*A_{j_z}E_{\ell_z}^*)$ is $gh$. The induction is finished. The desired lemma follows.
\end{proof}
\begin{lem}\label{L;Case5}
Let $a\in \mathbb{N}_{0}$ and assume that $R_{i_b}, R_{j_b}, R_{\ell_b}\in S$ for all $b\in [a]$. If we have $p_{i_bj_b}^{\ell_b}=1$ and $\ell_e=i_{e+1}$ for all $b\in [a]$ and $e\in [a-1]$, then the $r$-row of $\prod_{z=0}^{a}(E_{i_z}^*A_{j_z}E_{\ell_z}^*)$ has at least one nonzero entry for any $r\in xR_{i_0}$. Moreover, the $c$-column of $\prod_{z=0}^{a}(E_{i_z}^*A_{j_z}E_{\ell_z}^*)$ has at least one nonzero entry for any $c\in xR_{\ell_a}$.
\end{lem}
\begin{proof}
We work by induction. Let us consider the base case $a=0$. Since $p_{i_0j_0}^{\ell_0}=1$, by Lemmas \ref{L;Intersectionnumber} (iv) and \ref{L;generalresults} (i), both the $r$-row and the $c$-column of $E_{i_0}^*A_{j_0}E_{\ell_0}^*$ have at least one nonzero entry for any $r\in xR_{i_0}$ and $c\in xR_{\ell_0}$. The base case is checked.

Assume that $a>0$. As $p_{i_bj_b}^{\ell_b}=1$ and $\ell_e=i_{e+1}$ for all $b\in [a]$ and $e\in [a-1]$, for any $m\in [a-1]$, we can assume that the following assertions hold for $\prod_{z=0}^{m}(E_{i_z}^*A_{j_z}E_{\ell_z}^*)$.
\begin{enumerate}[(i)]
\item For any $f\in xR_{i_0}$, the $f$-row of $\prod_{z=0}^{m}(E_{i_z}^*A_{j_z}E_{\ell_z}^*)$ has at least one nonzero entry.
\item For any $g\in xR_{\ell_m}$, the $g$-column of $\prod_{z=0}^{m}(E_{i_z}^*A_{j_z}E_{\ell_z}^*)$ has at least one nonzero entry.
\end{enumerate}

By the inductive hypothesis (i) and \eqref{Eq:preliminary2}, for any $r\in xR_{i_0}$, there exists $h\in xR_{\ell_{a-1}}$ such that the $(r,h)$-entry of $\prod_{z=0}^{a-1}(E_{i_z}^*A_{j_z}E_{\ell_z}^*)$ is not zero. Write $k$ for the $(r,h)$-entry of $\prod_{z=0}^{a-1}(E_{i_z}^*A_{j_z}E_{\ell_z}^*)$. So $k\neq 0$. As $\ell_{a-1}=i_a$, observe that $h\in xR_{i_a}$. Since $p_{i_aj_a}^{\ell_a}=1$, by Lemmas \ref{L;Intersectionnumber} (iv) and \ref{L;generalresults} (i), the $h$-row of $E_{i_a}^*A_{j_a}E_{\ell_a}^*$ has at least one entry that equals one. By \eqref{Eq:preliminary2}, let $n\in xR_{\ell_a}$, where the $(h,n)$-entry of $E_{i_a}^*A_{j_a}E_{\ell_a}^*$ is one. Since $p_{i_aj_a}^{\ell_a}=1$ and $n\in xR_{\ell_a}$, by Lemma \ref{L;generalresults} (i) and \eqref{Eq:preliminary2}, notice that the $(q,n)$-entry of $E_{i_a}^*A_{j_a}E_{\ell_a}^*$ is zero for all $q\in X\setminus\{h\}$.

As $\prod_{z=0}^{a}(E_{i_z}^*A_{j_z}E_{\ell_z}^*)\!=\!(\prod_{z=0}^{a-1}(E_{i_z}^*A_{j_z}E_{\ell_z}^*))(E_{i_a}^*A_{j_a}E_{\ell_a}^*)$ and we have known that the $(q,n)$-entry of $E_{i_a}^*A_{j_a}E_{\ell_a}^*$ is zero for all $q\in X\setminus\{h\}$, observe that the $(r,n)$-entry of $\prod_{z=0}^{a}(E_{i_z}^*A_{j_z}E_{\ell_z}^*)$ equals the product of the $(r,h)$-entry of $\prod_{z=0}^{a-1}(E_{i_z}^*A_{j_z}E_{\ell_z}^*)$ and the $(h,n)$-entry of $E_{i_a}^*A_{j_a}E_{\ell_a}^*$. So the $(r,n)$-entry of $\prod_{z=0}^{a}(E_{i_z}^*A_{j_z}E_{\ell_z}^*)$ is $k$. As $k\neq 0$, we get that the $r$-row of $\prod_{z=0}^{a}(E_{i_z}^*A_{j_z}E_{\ell_z}^*)$ has at least one nonzero entry.

For any $c\in xR_{\ell_a}$, by the equality $p_{i_aj_a}^{\ell_a}=1$, Lemma \ref{L;generalresults} (i), and \eqref{Eq:preliminary2},
there exists $u\in xR_{i_a}$ such that the $(u,c)$-entry of $E_{i_a}^*A_{j_a}E_{\ell_a}^*$ equals one. Moreover, we also note that the $(v,c)$-entry of $E_{i_a}^*A_{j_a}E_{\ell_a}^*$ is zero for all $v\in X\setminus\{u\}$. Since $u\in xR_{i_a}$ and $\ell_{a-1}=i_a$, notice that $u\in xR_{\ell_{a-1}}$. By the inductive hypothesis (ii) and \eqref{Eq:preliminary2}, there is $w\in xR_{i_0}$ such that the $(w,u)$-entry of $\prod_{z=0}^{a-1}(E_{i_z}^*A_{j_z}E_{\ell_z}^*)$ is not zero. Let $y$ be the $(w,u)$-entry of $\prod_{z=0}^{a-1}(E_{i_z}^*A_{j_z}E_{\ell_z}^*)$. So $y\neq 0$.

As $\prod_{z=0}^{a}(E_{i_z}^*A_{j_z}E_{\ell_z}^*)\!=\!(\prod_{z=0}^{a-1}(E_{i_z}^*A_{j_z}E_{\ell_z}^*))(E_{i_a}^*A_{j_a}E_{\ell_a}^*)$ and we have known that the $(v,c)$-entry of $E_{i_a}^*A_{j_a}E_{\ell_a}^*$ is zero for all $v\in X\setminus\{u\}$, observe that the $(w,c)$-entry of $\prod_{z=0}^{a}(E_{i_z}^*A_{j_z}E_{\ell_z}^*)$ equals the product of the $(w,u)$-entry of $\prod_{z=0}^{a-1}(E_{i_z}^*A_{j_z}E_{\ell_z}^*)$ and the $(u,c)$-entry of $E_{i_a}^*A_{j_a}E_{\ell_a}^*$. So the $(w,c)$-entry of $\prod_{z=0}^{a}(E_{i_z}^*A_{j_z}E_{\ell_z}^*)$ is $y$. As $y\neq 0$, we get that the $c$-column of $\prod_{z=0}^{a}(E_{i_z}^*A_{j_z}E_{\ell_z}^*)$ has at least one nonzero entry. The induction is finished.
The desired lemma follows.
\end{proof}
\begin{lem}\label{L;Case6}
Let $a\in \mathbb{N}_{0}$ and assume that $R_{i_b}, R_{j_b}, R_{\ell_b}\in S$ for all $b\in [a]$. If we have $k_{i_0}>1=p_{i_bj_b}^{\ell_b}$ for all $b\in [a]$, for any $r\in xR_{i_0}$, there exists $u\in xR_{\ell_a}$ such that the $(r,u)$-entry of $\prod_{z=0}^{a}(E_{i_z}^*A_{j_z}E_{\ell_z}^*)$ is zero. Moreover, for any $c\in xR_{\ell_a}$, there exists $v\in xR_{i_0}$ such that the $(v,c)$-entry of $\prod_{z=0}^{a}(E_{i_z}^*A_{j_z}E_{\ell_z}^*)$ is zero.
\end{lem}
\begin{proof}
We work by induction. Let us consider the base case $a=0$. For any $r\in xR_{i_0}$, assume that the $(r,f)$-entry of $E_{i_0}^*A_{j_0}E_{\ell_0}^*$ equals one for all $f\in xR_{\ell_0}$. As we have  $|xR_{i_0}|=k_{i_0}>1$, there exists $g\in X\setminus\{r\}$ such that $g\in xR_{i_0}$. By Lemma \ref{L;generalresults} (i) and \eqref{Eq:preliminary2}, notice that the $(g,f)$-entry of $E_{i_0}^*A_{j_0}E_{\ell_0}^*$ is one for all $f\in xR_{\ell_0}$. Therefore there exists $h\in xR_{\ell_0}$ such that the $h$-column of $E_{i_0}^*A_{j_0}E_{\ell_0}^*$ has at least two entries that equal one. Since $p_{i_0j_0}^{\ell_0}=1$, by Lemma \ref{L;generalresults} (i), the $h$-column of $E_{i_0}^*A_{j_0}E_{\ell_0}^*$ has exactly one entry that equals one, which yields a contradiction. So there is $\tilde{u}\in xR_{\ell_0}$ such that the $(r,\tilde{u})$-entry of $E_{i_0}^*A_{j_0}E_{\ell_0}^*$ is zero.

For any $c\in xR_{\ell_0}$, by the equality $p_{i_0j_0}^{\ell_0}=1$, Lemma \ref{L;generalresults} (i), and \eqref{Eq:preliminary2}, there exists $k\in xR_{i_0}$ such that the $(k,c)$-entry of $E_{i_0}^*A_{j_0}E_{\ell_0}^*$ equals one. Since $|xR_{i_0}|=k_{i_0}>1$, notice that there is $\tilde{v}\in X\setminus\{k\}$ such that $\tilde{v}\in xR_{i_0}$. As $p_{i_0j_0}^{\ell_0}=1$, the $(\tilde{v},c)$-entry of $E_{i_0}^*A_{j_0}E_{\ell_0}^*$ is zero by Lemma \ref{L;generalresults} (i) and \eqref{Eq:preliminary2}. The base case is checked.

Assume that $a>0$. As $k_{i_0}>1=p_{i_bj_b}^{\ell_b}$ for all $b\in [a]$, for any $m\in [a-1]$, we can assume that the following assertions hold for $\prod_{z=0}^{m}(E_{i_z}^*A_{j_z}E_{\ell_z}^*)$.
\begin{enumerate}[(i)]
\item For any $\hat{r}\in xR_{i_0}$, there is $\hat{u}\in xR_{\ell_m}$ such that the $(\hat{r},\hat{u})$-entry of $\prod_{z=0}^{m}(E_{i_z}^*A_{j_z}E_{\ell_z}^*)$ is zero.
\item For any $\hat{c}\in xR_{\ell_m}$, there is $\hat{v}\in xR_{i_0}$ such that the $(\hat{v},\hat{c})$-entry of $\prod_{z=0}^{m}(E_{i_z}^*A_{j_z}E_{\ell_z}^*)$ is zero.
\end{enumerate}

If $\prod_{z=0}^{a}(E_{i_z}^*A_{j_z}E_{\ell_z}^*)=O$, then the desired lemma holds. So we can assume further that $\prod_{z=0}^{a}(E_{i_z}^*A_{j_z}E_{\ell_z}^*)\neq O$, which implies that $\ell_{a-1}=i_a$.

For any $r\in xR_{i_0}$, by the inductive hypothesis (i), there is $n\in xR_{\ell_{a-1}}$ such that the $(r,n)$-entry of $\prod_{z=0}^{a-1}(E_{i_z}^*A_{j_z}E_{\ell_z}^*)$ is zero. Since $\ell_{a-1}=i_a$ and $n\in xR_{\ell_{a-1}}$, notice that $n\in xR_{i_a}$. Since $p_{i_aj_a}^{\ell_a}=1$, by Lemma \ref{L;Case5} and \eqref{Eq:preliminary2}, there is $u\in xR_{\ell_a}$ such that the $(n,u)$-entry of $E_{i_a}^*A_{j_a}E_{\ell_a}^*$ is one. Since $p_{i_aj_a}^{\ell_a}=1$, by Lemma \ref{L;generalresults} (i) and \eqref{Eq:preliminary2}, observe that the $(q,u)$-entry of $E_{i_a}^*A_{j_a}E_{\ell_a}^*$ is zero for all $q\in X\setminus\{n\}$.

As $\prod_{z=0}^{a}(E_{i_z}^*A_{j_z}E_{\ell_z}^*)\!=\!(\prod_{z=0}^{a-1}(E_{i_z}^*A_{j_z}E_{\ell_z}^*))(E_{i_a}^*A_{j_a}E_{\ell_a}^*)$ and we have known that the $(q,u)$-entry of $E_{i_a}^*A_{j_a}E_{\ell_a}^*$ is zero for all $q\in X\setminus\{n\}$, observe that the $(r,u)$-entry of $\prod_{z=0}^{a}(E_{i_z}^*A_{j_z}E_{\ell_z}^*)$ equals the product of the $(r,n)$-entry of $\prod_{z=0}^{a-1}(E_{i_z}^*A_{j_z}E_{\ell_z}^*)$ and the $(n,u)$-entry of $E_{i_a}^*A_{j_a}E_{\ell_a}^*$. Since the $(r,n)$-entry of $\prod_{z=0}^{a-1}(E_{i_z}^*A_{j_z}E_{\ell_z}^*)$ is zero, the $(r,u)$-entry of $\prod_{z=0}^{a}(E_{i_z}^*A_{j_z}E_{\ell_z}^*)$ is zero.

For any $c\in xR_{\ell_a}$, as $p_{i_aj_a}^{\ell_a}=1$, Lemma \ref{L;Case5} and \eqref{Eq:preliminary2} tell us that there is $w\in xR_{i_a}$ such that the $(w,c)$-entry of $E_{i_a}^*A_{j_a}E_{\ell_a}^*$ equals one. Moreover, by Lemma \ref{L;generalresults} (i) and \eqref{Eq:preliminary2}, the $(y,c)$-entry of $E_{i_a}^*A_{j_a}E_{\ell_a}^*$ is zero for all $y\in X\setminus\{w\}$. Since $\ell_{a-1}=i_a$ and $w\in xR_{i_a}$, note that $w\in xR_{\ell_{a-1}}$. By the inductive hypothesis (ii), there is $v\in xR_{i_0}$ such that $(v,w)$-entry of $\prod_{z=0}^{a-1}(E_{i_z}^*A_{j_z}E_{\ell_z}^*)$ is zero.

Since $\prod_{z=0}^{a}(E_{i_z}^*A_{j_z}E_{\ell_z}^*)\!=\!(\prod_{z=0}^{a-1}(E_{i_z}^*A_{j_z}E_{\ell_z}^*))(E_{i_a}^*A_{j_a}E_{\ell_a}^*)$ and we have known that the $(y,c)$-entry of $E_{i_a}^*A_{j_a}E_{\ell_a}^*$ is zero for all $y\in X\setminus\{w\}$, observe that the $(v,c)$-entry of $\prod_{z=0}^{a}(E_{i_z}^*A_{j_z}E_{\ell_z}^*)$ equals the product of the $(v,w)$-entry of $\prod_{z=0}^{a-1}(E_{i_z}^*A_{j_z}E_{\ell_z}^*)$ and the $(w,c)$-entry of $E_{i_a}^*A_{j_a}E_{\ell_a}^*$. Since $(v,w)$-entry of $\prod_{z=0}^{a-1}(E_{i_z}^*A_{j_z}E_{\ell_z}^*)$ is zero, notice that the $(v,c)$-entry of $\prod_{z=0}^{a}(E_{i_z}^*A_{j_z}E_{\ell_z}^*)$ is zero. The induction is finished. The desired lemma follows.
\end{proof}
The following lemma strengthens Lemmas \ref{L;Case5} and \ref{L;Case6}.
\begin{lem}\label{L;Case7}
Let $a\in \mathbb{N}_{0}$ and assume that $R_{i_b}, R_{j_b}, R_{\ell_b}\in S$ for all $b\in [a]$. If we have $p_{i_bj_b}^{\ell_b}=1$ for all $b\in [a]$, then $\prod_{z=0}^{a}(E_{i_z}^*A_{j_z}E_{\ell_z}^*)$ is a $(0,1)$-matrix.
\end{lem}
\begin{proof}
We work by induction. Observe that the base case $a=0$ is checked by \eqref{Eq:preliminary2}. Assume that $a>0$. As $p_{i_bj_b}^{\ell_b}=1$ for all $b\in [a]$, we can assume that $\prod_{z=0}^{m}(E_{i_z}^*A_{j_z}E_{\ell_z}^*)$ is a $(0,1)$-matrix for any $m\in [a-1]$. By the inductive hypothesis, $\prod_{z=0}^{a-1}(E_{i_z}^*A_{j_z}E_{\ell_z}^*)$ is a $(0,1)$-matrix.

Since $\prod_{z=0}^{a-1}(E_{i_z}^*A_{j_z}E_{\ell_z}^*)$ is a $(0,1)$-matrix, every entry of $\prod_{z=0}^{a-1}(E_{i_z}^*A_{j_z}E_{\ell_z}^*)$ is either zero or one. By the equality $p_{i_aj_a}^{\ell_a}=1$, Lemma \ref{L;generalresults} (i), and \eqref{Eq:preliminary2}, for any $c\in xR_{\ell_a}$, the $c$-column of $E_{i_a}^*A_{j_a}E_{\ell_a}^*$ has exactly one nonzero entry that equals one. For any $e\in X\setminus xR_{\ell_a}$, by \eqref{Eq:preliminary2}, note that every entry of the $e$-column of $E_{i_a}^*A_{j_a}E_{\ell_a}^*$ is zero. As $\prod_{z=0}^{a}(E_{i_z}^*A_{j_z}E_{\ell_z}^*)=\!(\prod_{z=0}^{a-1}(E_{i_z}^*A_{j_z}E_{\ell_z}^*))(E_{i_a}^*A_{j_a}E_{\ell_a}^*)$, it is obvious to see that every entry of $\prod_{z=0}^{a}(E_{i_z}^*A_{j_z}E_{\ell_z}^*)$ is either zero or one. The induction is finished. The desired lemma follows.
\end{proof}
For further discussion, for any $(a,b)\in X\times X$, recall that $E_{ab}$ is the $(0,1)$-matrix of $M_X(\F)$ whose unique nonzero entry is the $(a,b)$-entry. We are now ready to deduce the following corollary.
\begin{cor}\label{C;badpairform}
Let $a\in \mathbb{N}_{0}$ and assume that $R_{i_b}, R_{j_b}, R_{\ell_b}\in S$ for all $b\in [a]$. If we have $p_{i_bj_b}^{\ell_b}=1$, $k_{i_b}=k_{\ell_b}=2$, and $\ell_c=i_{c+1}$ for all $b\in [a]$ and $c\in [a-1]$, then we have $\prod_{z=0}^{a}(E_{i_z}^*A_{j_z}E_{\ell_z}^*)=E_{u_1v_1}+E_{u_2v_2}$ or $\prod_{z=0}^{a}(E_{i_z}^*A_{j_z}E_{\ell_z}^*)=E_{u_1v_2}+E_{u_2v_1}$, where $xR_{i_0}=\{u_1, u_2\}$ and $xR_{\ell_a}=\{v_1, v_2\}$.
\end{cor}
\begin{proof}
By the hypotheses and \eqref{Eq:preliminary2}, all possible nonzero entries of $\prod_{z=0}^{a}(E_{i_z}^*A_{j_z}E_{\ell_z}^*)$ are the $(u_1, v_1)$-entry, the $(u_1, v_2)$-entry, the $(u_2, v_1)$-entry, and the $(u_2, v_2)$-entry. So the desired corollary follows by the hypotheses and Lemmas \ref{L;Case4}, \ref{L;Case5}, \ref{L;Case6}, \ref{L;Case7}.
\end{proof}
For further discussion, we present the required notation.
\begin{nota}\label{N;notation}
\em Let
$\mathcal{S}_a=\{\prod_{b=0}^a(E_{i_b}^*A_{j_b}E_{\ell_b}^*): R_{i_b}, R_{j_b}, R_{\ell_b}\in S\ \text{for all}\ b\in[a]\}\setminus\{O\}$
for any $a\in \mathbb{N}_0$ and set $\mathcal{S}_{-1}=\varnothing$. For any $c\in \mathbb{N}_0\cup\{-1\}$, put $\mathcal{T}_c=\langle \mathcal{S}_c\rangle_\F\subseteq \mathcal{T}$ and let $\sum_{e=-1}^c\mathcal{T}_e=\langle\bigcup_{e=-1}^c\mathcal{S}_e\rangle_\F$. Write $\sum_{f\in \mathbb{N}_0}\mathcal{T}_f$ for $\langle \bigcup_{f\in \mathbb{N}_0}\mathcal{S}_f\rangle_\F$.
\end{nota}
\begin{lem}\label{L;Linearsolution}
The following assertions hold.
\begin{enumerate}[(i)]
\item [\em (i)] We have $\mathcal{T}=\sum_{f\in \mathbb{N}_0}\mathcal{T}_f$ as $\F$-algebras.
\item [\em (ii)] Let $a\in \mathbb{N}_0$ and $M\in \mathcal{S}_a$. Then there exist $c, a_0, a_1,\ldots, a_c\in \mathbb{N}_0$ such that
    $$M=\sum_{b=0}^ce_bM_b,$$ where we have $0\neq e_b\in \F$, $M_b\in \mathcal{S}_{a_b}$, and $M_b\notin\sum_{h=-1}^{a_{b}-1}\mathcal{T}_h$ for any $b\in [c]$.
\item [\em (iii)] Let $M\in \mathcal{T}$. Then there exist $c, a_0, a_1,\ldots, a_c\in \mathbb{N}_0$ such that
    \begin{align*}
    M=\sum_{b=0}^ce_bM_b,
    \end{align*}
   where we have $0\neq e_b\in \F$, $M_b\in \mathcal{S}_{a_b}$, and $M_b\notin\sum_{h=-1}^{a_{b}-1}\mathcal{T}_h$ for any $b\in [c]$. Moreover, let $b\in [c]$. If we have $a_b>0$ and $M_b=\prod_{u=0}^{a_b}(E_{i_u}^*A_{j_u}E_{\ell_u}^*)$, then we have $\min\{ k_{i_u}, k_{j_u}, k_{\ell_u}\}>1$, $p_{i_uj_u}^{\ell_u}>0$, $\ell_v=i_{v+1}$, $k_{i_v}k_{j_v}\neq p_{i_vj_v}^{\ell_v}k_{\ell_v}$ for any $u\in [a_b]$ and $v\in [a_b-1]$.
\end{enumerate}
\end{lem}
\begin{proof}
For (i), by the definition of $\sum_{f\in \mathbb{N}_0}\mathcal{T}_f$, $\sum_{f\in \mathbb{N}_0}\mathcal{T}_f$ is a unital $\F$-subalgebra of $\mathcal{T}$. Moreover, as $\mathcal{T}_0\subseteq\sum_{f\in \mathbb{N}_0}\mathcal{T}_f$, notice that $A_e, E_e^*\in\mathcal{T}_0\subseteq\sum_{f\in \mathbb{N}_0}\mathcal{T}_f$ for all $R_e\in S$. (i) is now shown by the definition of $\mathcal{T}$.

For (ii), we work by induction. When $a=0$, since $M\neq O$, we set $c=a_0=0$, $M_0=M$, and $e_0=1$. The desired assertion thus holds. The base case is checked. We now assume that $a>0$ and the desired assertion holds for every member of $\bigcup_{u=0}^{a-1}\mathcal{S}_u$. If $M\notin \sum_{h=-1}^{a-1}\mathcal{T}_h$, set $c=0$ and $a_0=a$. The desired assertion holds since we can set $M_0=M$ and $e_0=1$. If $M\in\sum_{h=-1}^{a-1}\mathcal{T}_h$, since $M\neq O$ and (i) holds, there exist $\ell\in \mathbb{N}_0$ and $g_0, g_1,\ldots, g_{\ell}\in [a-1]$ such that
\begin{align*}\label{Eq:7}
\tag{3.7}M=\sum_{m=0}^\ell z_mM(m),
\end{align*}
where we have $0\neq z_m\in\F$ and $M(m)\in \mathcal{S}_{g_m}$ for any $m\in [\ell]$. As $M(m)\in \mathcal{S}_{g_m}$ and $g_m<a$ for any $m\in [\ell]$, the inductive hypothesis holds for any $M(m)$. (ii) is now proved by the induction and \eqref{Eq:7}.

If $M\neq O$, the first assertion of (iii) follows by (i) and (ii). Since $E_0^*\in \mathcal{S}_0$ and $O=E_0^*-E_0^*$, the first assertion of (iii) thus follows. If there is $y\in [a_b]$ such that $\min\{k_{i_y}, k_{j_y}, k_{\ell_y}\}=1$, $M_b\in \sum_{h=-1}^{a_b-1}\mathcal{T}_h$ by Lemmas \ref{L;Case1} and \ref{L;Case2}. If there is $z\in [a_b-1]$ such that $k_{i_z}k_{j_z}=p_{i_zj_z}^{\ell_z}k_{\ell_z}$, then $M_b\in \sum_{h=-1}^{a_b-1}\mathcal{T}_h$ by Lemma \ref{L;Case3} (i). If there is $y\in [a_b]$ or $z\in [a_b-1]$ such that $p_{i_yj_y}^{\ell_y}=0$ or $\ell_z\neq i_{z+1}$, then $M_b=O$. By these mentioned facts and the first assertion of (iii), the second assertion of (iii) follows.
\end{proof}
\begin{eg}\label{E;example1}
\em We now use an example to illustrate Lemma \em \ref{L;Linearsolution} (iii). \em Assume that $X=[28]\setminus\{0\}$ and $x=1$. Assume further that $S$ is the scheme of order $28$, No. $175$ in \cite{HM}. Then $S=\bigcup_{a=0}^{15}\{R_a\}$. By computation, notice that
\begin{align*}
 O\neq E_4^*A_{12}E_8^*=E_4^*A_7E_9^*A_{13}E_8^*+E_4^*A_7E_9^*A_{15}E_8^*\in \mathcal{T},
\end{align*}
where we have $E_4^*A_7E_9^*A_{13}E_8^*, E_4^*A_7E_9^*A_{15}E_8^*\notin \sum_{h=-1}^0\mathcal{T}_h$, $p_{47}^9=p_{9(13)}^8=p_{9(15)}^8=1$, and $k_4=k_7=k_8=k_9=k_{13}=k_{15}=2$.
\end{eg}
We conclude this section with the following corollaries.
\begin{cor}\label{C;twovaluedecomposition}
Let $1<k\in \mathbb{N}$ and $M\in \mathcal{T}$. If $k_a\in\{1, k\}$ for all $R_a\in S$, then there exist $c, a_0, a_1,\ldots, a_c\in \mathbb{N}_0$ such that
\begin{align*}
M=\sum_{b=0}^ce_bM_b,
\end{align*}
where we have $0\neq e_b\in \F$, $M_b\in \mathcal{S}_{a_b}$, and $M_b\notin\sum_{h=-1}^{a_{b}-1}\mathcal{T}_h$ for any $b\in [c]$. Moreover, let $b\in [c]$. If $a_b>0$ and $M_b=\prod_{u=0}^{a_b}(E_{i_u}^*A_{j_u}E_{\ell_u}^*)$, then we have $k_{i_u}=k_{j_u}=k_{\ell_u}=k$, $0<p_{i_uj_u}^{\ell_u}<k$, $\ell_v=i_{v+1}$ for any $u\in [a_b]$ and $v\in [a_b-1]$.
\end{cor}
\begin{proof}
The first assertion follows by Lemma \ref{L;Linearsolution} (iii). For the second assertion, set $w=a_b$. Since $k_a\in\{1,k\}$ for all $R_a\in S$, by Lemmas \ref{L;Linearsolution} (iii) and \ref{L;Intersectionnumber} (ii), it is enough to check that $p_{i_wj_w}^{\ell_w}<k$. By Lemma \ref{L;Intersectionnumber} (ii), we can assume that $p_{i_wj_w}^{\ell_w}=k$ to get a contradiction. By this assumption, Lemmas \ref{L;Linearsolution} (iii), and \ref{L;Intersectionnumber} (iii), we have $k_{i_w}=k_{j_w}=k_{\ell_w}=k$ and $k_{i_w}k_{j_w}=p_{i_wj_w}^{\ell_w}k_{\ell_w}$ as $k_a\in\{1,k\}$ for all $R_a\in S$. We thus have $M_b\in\sum_{h=-1}^{a_{b}-1}\mathcal{T}_h$ by Lemma \ref{L;Case3} (ii), which contradicts the proven first assertion. The second assertion thus follows.
\end{proof}
\begin{cor}\label{C;quasithindecomposiiton}
Assume that $S$ is a quasi-thin scheme and let $M\in \mathcal{T}$. Then there exist $c, a_0, a_1,\ldots, a_c\in \mathbb{N}_0$ such that
\begin{align*}
M=\sum_{b=0}^ce_bM_b,
\end{align*}
where we have $0\neq e_b\in \F$, $M_b\in \mathcal{S}_{a_b}$, and $M_b\notin\sum_{h=-1}^{a_{b}-1}\mathcal{T}_h$ for any $b\in [c]$. Moreover, let $b\in [c]$. If we have $a_b>0$ and $M_b=\prod_{u=0}^{a_b}(E_{i_u}^*A_{j_u}E_{\ell_u}^*)$, then we have the following assertions.
\begin{enumerate}[(i)]
\item [\em (i)] We have $k_{i_u}\!=\!k_{j_u}=\!k_{\ell_u}=2$, $p_{i_uj_u}^{\ell_u}\!\!=1$, $\ell_v=i_{v+1}$ for any $u\in [a_b]$ and $v\in [a_b-1]$.
\item [\em (ii)] We have $|R_{i_v'}R_{\ell_{v+1}}|=1$ for any $v\in [a_b-1]$.
\item [\em (iii)] We have $|R_{i_0'}R_{\ell_{a_b}}|=1$.
\end{enumerate}
\end{cor}
\begin{proof}
Since $S$ is a quasi-thin scheme, we only show (ii) and (iii) by Corollary \ref{C;twovaluedecomposition}.

For (ii), let $R_w\in R_{i_v'}R_{\ell_{v+1}}$ and notice that $p_{i_vw}^{\ell_{v+1}}>0$ by Lemma \ref{L;Intersectionnumber} (iv). We claim that (ii) follows if $p_{i_vw}^{\ell_{v+1}}\neq 1$. If $p_{i_vw}^{\ell_{v+1}}\neq 1$, since $S$ is a quasi-thin scheme, notice that $\max\{k_w, p_{i_vw}^{\ell_{v+1}}, p_{i_v'\ell_{v+1}}^w\}\leq 2$ by Lemma \ref{L;Intersectionnumber} (ii). So $p_{i_vw}^{\ell_{v+1}}=2$ as $p_{i_vw}^{\ell_{v+1}}\neq 1$. Moreover, since $k_{\ell_{v+1}}=2$ by (i), the equality  $p_{i_vw}^{\ell_{v+1}}=k_{\ell_{v+1}}=2$ implies that $k_w=p_{i_v'\ell_{v+1}}^w=2$ by Lemma \ref{L;Intersectionnumber} (iv). So $R_{i_v'}R_{\ell_{v+1}}=\{R_w\}$ by Lemma \ref{L;Intersectionnumber} (i) and (iii). The claim follows.

By the proven claim, we prove the inequality $p_{i_vw}^{\ell_{v+1}}\neq 1$. Suppose that $p_{i_vw}^{\ell_{v+1}}=1$. By (i) and the equality $p_{i_vw}^{\ell_{v+1}}=1$, notice that all assumptions of Corollary \ref{C;badpairform} hold for both $E_{i_v}^*A_{j_v}E_{\ell_{v}}^*A_{j_{v+1}}E_{\ell_{v+1}}^*$ and $E_{i_v}^*A_wE_{\ell_{v+1}}^*$. By Corollary \ref{C;badpairform} and \eqref{Eq:preliminary2}, we have
\begin{align*}
&E_{i_v}^*A_{j_v}E_{\ell_{v}}^*A_{j_{v+1}}E_{\ell_{v+1}}^*=E_{i_v}^*A_wE_{\ell_{v+1}}^*\ \text{or}\ \\
&E_{i_v}^*A_{j_v}E_{\ell_{v}}^*A_{j_{v+1}}E_{\ell_{v+1}}^*+E_{i_v}^*A_wE_{\ell_{v+1}}^*=E_{i_v}^*JE_{\ell_{v+1}}^*.
\end{align*}
In particular, we obtain that $E_{i_v}^*A_{j_v}E_{\ell_{v}}^*A_{j_{v+1}}E_{\ell_{v+1}}^*\in \mathcal{T}_0$ and $M_b\in \sum_{h=-1}^{a_b-1}\mathcal{T}_h$, which contradicts the fact $M_b\notin \sum_{h=-1}^{a_b-1}\mathcal{T}_h$. So $p_{i_vw}^{\ell_{v+1}}\neq 1$. (ii) thus follows.

For (iii), we set $q=\ell_{a_b}$. Since $S$ is a quasi-thin scheme, by Lemma \ref{L;Intersectionnumber} (v), we have $|R_{i_0'}R_{q}|\leq2$. Assume that $R_{i_0'}R_{q}=\{R_y, R_z\}$ and $y\neq z$. Notice that $p_{i_0'q}^y\neq 0\neq p_{i_0'q}^z$, which implies that $p_{i_0y}^q\neq 0\neq p_{i_0z}^q$ by Lemma \ref{L;Intersectionnumber} (iv). Since $k_{i_0}=2$ by (i), we thus deduce that $p_{i_0y}^q= p_{i_0z}^q=1$ by Lemma \ref{L;Intersectionnumber} (ii).

Since $M_b=\prod_{u=0}^{a_b}(E_{i_u}^*A_{j_u}E_{\ell_u}^*)\neq O$, notice that $\ell_r=i_{r+1}$ for any $r\in [a_b-1]$. Moreover, by (i) and the equality $p_{i_0y}^q= p_{i_0z}^q=1$, notice that all assumptions of Corollary \ref{C;badpairform} hold for $M_b$, $E_{i_0}^*A_yE_q^*$, and $E_{i_0}^*A_zE_q^*$. As $y\neq z$, it is not very difficult to see that $E_{i_0}^*A_yE_q^*\neq E_{i_0}^*A_zE_q^*$. We thus get that $M_b=E_{i_0}^*A_yE_q^*$ or $M_b=E_{i_0}^*A_zE_q^*$ by Corollary \ref{C;badpairform}. In particular,  $M_b\in\mathcal{T}_0$, which contradicts the fact $M_b\notin \sum_{h=-1}^{a_b-1}\mathcal{T}_h$. Therefore we have $|R_{i_0'}R_q|=1$. (iii) is proved.
\end{proof}
\section{$\F$-Dimensions of Terwilliger $\F$-algebras of quasi-thin schemes}
In this section, we show Theorem \ref{T;A}. Our strategy is to find an $\F$-basis of $\mathcal{T}$ when $S$ is a quasi-thin scheme. For our purpose, we recall Notation \ref{N;notation} and the definition of $E_{ab}$ in Subsection \ref{Sc:subsection}. We first propose a definition and the required notation.
\begin{defn}\label{D;badpair}
\em Let $a\in \mathbb{N}$ and assume that $R_{i_b}, R_{j_b}, R_{\ell_b}\in S$ for all $b\in [a]$. Then the pair $(i_0, \ell_a)$ is called a bad pair of $S$ if the following conditions hold.
\begin{enumerate}[(i)]
\item [\em (i)] We have $k_{i_b}=k_{\ell_b}=2$ and $p_{i_bj_b}^{\ell_b}=1$ for all $b\in [a]$;
\item [\em (ii)] We have $|R_{i_0'}R_{\ell_a}|=1$ and $\ell_c=i_{c+1}$ for all $c\in [a-1]$.
\end{enumerate}
In particular, for any $R_u, R_v\in S$, $(u,v)$ is a bad pair of $S$ only if $k_u=k_v=2$.
\end{defn}
\begin{nota}\label{N;notation1}
\em Let $\preceq$ be a fixed total order over $X$. For any $a, b\in X$, write $a\prec b$ if $a\preceq b$ and $a\neq b$. Let $\mathcal{R}=\{(g,h): R_g, R_h\in S,\ k_g=k_h=|R_{g'}R_h|=2\}$. Let $\mathcal{S}$ denote the set of all bad pairs of $S$ and set $\mathcal{U}=\mathcal{R}\cup \mathcal{S}$. If $(i,j)\in \mathcal{U}$, observe that $k_i=k_j=2$ by the definition of $\mathcal{U}$. Assume that $xR_i=\{u_1, u_2\}$ and $xR_j=\{v_1, v_2\}$, where $u_1\prec u_2$ and $v_1\prec v_2$. Set $B_{ij}=E_{u_1v_1}+E_{u_2v_2}$ and $\mathcal{V}=\{B_{uv}: (u,v)\in \mathcal{U}\}$. Let $\mathcal{W}=\{E_y^*JE_z^*: R_y, R_z\in S\}$ and $\mathcal{B}=\mathcal{V}\cup\mathcal{W}$. By Lemma \em \ref{L;generalresults} (iii), $|\mathcal{W}|=(d+1)^2$.
\end{nota}
\begin{lem}\label{L;Badpairlemma}
The following assertions hold.
\begin{enumerate}[(i)]
\item [\em (i)] We have $|\mathcal{U}|=|\mathcal{R}|+|\mathcal{S}|$.
\item [\em (ii)] If $B_{ab}\in \mathcal{V}$, we have $E_g^*B_{ab}E_h^*=\delta_{ga}\delta_{bh}B_{ab}$ for any $R_g, R_h\in S$, where $\delta_{\alpha\beta}$ is the Kronecker delta of integers $\alpha, \beta$ whose values are in $\F$.
\item [\em (iii)] If $B_{ab}\in \mathcal{V}$, we have $B_{ab}\in \mathcal{T}\setminus \langle\mathcal{W}\rangle_\F$.
\item [\em (iv)] We have $|\mathcal{B}|=|\mathcal{V}|+|\mathcal{W}|=|\mathcal{U}|+|\mathcal{W}|=(d+1)^2+|\mathcal{R}|+|\mathcal{S}|$.
\end{enumerate}
\end{lem}
\begin{proof}
For (i), we may assume that $\mathcal{U}, \mathcal{R}, \mathcal{S}$ are not empty. As $\mathcal{U}=\mathcal{R}\cup\mathcal{S}$, it suffices to show $\mathcal{R}\cap \mathcal{S}=\varnothing$. For any $(r,s)\in \mathcal{R}$, note that $|R_{r'}R_s|=2$. For any $(u,v)\in \mathcal{S}$, by Definition \ref{D;badpair} (ii), we have $|R_{u'}R_v|=1$. So we get $\mathcal{R}\cap\mathcal{S}=\varnothing$. (i) thus follows.

Assume that $B_{ab}\in\mathcal{V}$. So $k_a=k_b=2$ by the definition of $\mathcal{V}$. So we can set $xR_a=\{u_1, u_2\}$ and $xR_b=\{v_1, v_2\}$, where $u_1\prec u_2$ and $v_1\prec v_2$. We list two cases.
\begin{enumerate}[\text{Case} 1:]
\item $(a,b)\in \mathcal{R}$.
\end{enumerate}
In this case, we have $k_a=k_b=|R_{a'}R_b|=2$. \!By Lemmas \ref{L;Intersectionnumber} (iv), (ii), and \ref{L;generalresults} (i), there exists $R_c\in S$ such that $E_a^*A_cE_b^*\neq O$ and $p_{ac}^b=1$. So we can deduce that $E_a^*A_cE_b^*=E_{u_1v_1}+E_{u_2v_2}$ or $E_a^*A_cE_b^*=E_{u_1v_2}+E_{u_2v_1}$ by Corollary \ref{C;badpairform}. In particular, by the definition of $B_{ab}$ and \eqref{Eq:preliminary2}, we have $B_{ab}=E_a^*A_cE_b^*$ or $B_{ab}=E_a^*JE_b^*-E_a^*A_cE_b^*$.
\begin{enumerate}[\text{Case} 2:]
\item $(a,b)\in \mathcal{S}$.
\end{enumerate}
By Definition \ref{D;badpair}, there are $e\in \mathbb{N}$ and $\prod_{m=0}^e(E_{i_m}^*A_{j_m}E_{\ell_m}^*)$ such that $i_0=a$, $\ell_e=b$, $k_{i_m}=k_{\ell_m}=2$ and $p_{i_mj_m}^{\ell_m}=1$ for all $m\in [e]$, $\ell_{n}=i_{n+1}$ for all $n\in [e-1]$. We have $\prod_{m=0}^e(E_{i_m}^*A_{j_m}E_{\ell_m}^*)=E_{u_1v_1}+E_{u_2v_2}$ or $\prod_{m=0}^e(E_{i_m}^*A_{j_m}E_{\ell_m}^*)=E_{u_1v_2}+E_{u_2v_1}$ by Corollary \ref{C;badpairform}. In particular, according to the definition of $B_{ab}$ and \eqref{Eq:preliminary2}, we can deduce that  $B_{ab}=\prod_{m=0}^e(E_{i_m}^*A_{j_m}E_{\ell_m}^*)$ or $B_{ab}=E_a^*JE_b^*-\prod_{m=0}^e(E_{i_m}^*A_{j_m}E_{\ell_m}^*)$.

For (ii), the desired equality is shown by the conclusions of Cases 1 and 2.

For (iii), by the conclusions of Cases 1 and 2, notice that $B_{ab}\in \mathcal{T}$. It is enough to show that $B_{ab}$ can not be written as an $\F$-linear combination of the members of $\mathcal{W}$. Suppose that $B_{ab}=\sum_{y=0}^d\sum_{z=0}^dc_{yz}E_y^*JE_z^*$, where $c_{yz}\in \F$ for any $R_y, R_z\in S$. By (ii), we have $B_{ab}=E_a^*B_{ab}E_b^*=c_{ab}E_a^*JE_b^*$. Recall that we have already known that $k_a=k_b=2$. By the definition of $B_{ab}$ and \eqref{Eq:preliminary2}, we thus deduce that
$$ E_{u_1v_1}+E_{u_2v_2}=B_{ab}=E_a^*B_{ab}E_b^*=c_{ab}E_a^*JE_b^*=c_{ab}(E_{u_1v_1}+E_{u_1v_2}+E_{u_2v_1}+E_{u_2v_2}),$$
which is an obvious contradiction. Therefore $B_{ab}$ can not be written as an $\F$-linear combination of the members of $\mathcal{W}$. (iii) thus follows.

For (iv), we have $\mathcal{V}\cap\mathcal{W}=\varnothing$ by (iii). So $|\B|=|\mathcal{V}|+|\mathcal{W}|$ as $\B=\mathcal{V}\cup\mathcal{W}$. By (i), it suffices to prove $|\mathcal{U}|=|\mathcal{V}|$. We may assume further that $\mathcal{U}\neq \varnothing$. Let $\phi$ be the surjective map from $\mathcal{U}$ to $\mathcal{V}$ that sends every $(i,j)$ to $B_{ij}$. For any $B_{ab}, B_{qw}\in \mathcal{V}$, by (ii), notice that $B_{ab}=B_{qw}$ if and only if $a=q$ and $b=w$, which implies that $\phi$ is a bijection from $\mathcal{U}$ to $\mathcal{V}$. We thus have $|\mathcal{U}|=|\mathcal{V}|$. (iv) thus follows.
\end{proof}
\begin{lem}\label{L;Linearlyindependent}
The set $\B$ is an $\F$-linearly independent subset of $\mathcal{T}$.
\end{lem}
\begin{proof}
As $\B=\mathcal{V}\cup \mathcal{W}$ and $\mathcal{W}\subseteq\mathcal{T}$, we have $\B\subseteq \mathcal{T}$ as $\mathcal{V}\subseteq \mathcal{T}$ by Lemma \ref{L;Badpairlemma} (iii). It suffices to check that $\B$ is linearly independent over $\F$. By Lemma \ref{L;generalresults} (iii), we may assume further that $\mathcal{V}\neq \varnothing$. Assume that we have
\begin{align*}\label{Eq:8}
\tag{4.1}\sum_{i=0}^d\sum_{j=0}^dc_{ij}E_i^*JE_j^*+\sum_{(a,b)\in \mathcal{U}}e_{ab}B_{ab}=O,
\end{align*}
where we have $c_{ij}, e_{ab}\in \F$ for any $R_i, R_j\in S$ and $(a,b)\in \mathcal{U}$. By \eqref{Eq:8} and Lemma \ref{L;Badpairlemma} (ii), for any $(g,h)\in \mathcal{U}$, we have $O=E_g^*OE_h^*=c_{gh}E_g^*JE_h^*+e_{gh}B_{gh}$, which implies that $e_{gh}=0$ by Lemma \ref{L;Badpairlemma} (iii). By Lemma \ref{L;generalresults} (iii), we also have $c_{ij}=0$ for any $R_i, R_j\in S$. So $\B$ is linearly independent over $\F$. The desired lemma is shown.
\end{proof}
We collect some results of quasi-thin schemes as follows.
\begin{lem}\label{L;Quasi-thin}
Assume that $S$ is a quasi-thin scheme.
\begin{enumerate}[(i)]
\item [\em (i)] The set $\mathcal{S}$ is empty if and only if there do not exist $R_u, R_v, R_w, R_y, R_z\in S$ such that $k_u=k_v=k_w=k_y=k_z=2$ and $p_{uv}^w=p_{wy}^z=|R_{u'}R_z|=1$.
\item [\em (ii)] If $M\in \mathcal{T}$, then $M=\sum_{N\in \mathcal{B}}c_N N,$
where $c_N\in \F$ for any $N\in \mathcal{B}$.
\end{enumerate}
\end{lem}
\begin{proof}
For (i), assume that $\mathcal{S}=\varnothing$. If there exist $R_u, R_v, R_w, R_y, R_z\in S$ such that $k_u=k_v=k_w=k_y=k_z=2$ and $p_{uv}^w=p_{wy}^z=|R_{u'}R_z|=1$, by Definition \ref{D;badpair}, notice that $(u,z)\in \mathcal{S}$ as Definition \ref{D;badpair} (i) and (ii) hold. So we have $\mathcal{S}\neq \varnothing$, which contradicts the assumption $\mathcal{S}=\varnothing$. One direction is shown.

For the other direction, assume that there do not exist $R_u, R_v, R_w, R_y, R_z\in S$ such that $k_u=k_v=k_w=k_y=k_z=2$ and $p_{uv}^w=p_{wy}^z=|R_{u'}R_z|=1$. If $\mathcal{S}\neq \varnothing$, we choose  $(a,b)\in \mathcal{S}$ and notice that $B_{ab}\in \mathcal{V}\subseteq \mathcal{T}$ by Lemma \ref{L;Badpairlemma} (iii). Moreover, we also have $|R_{a'}R_b|=1$ by Definition \ref{D;badpair} (ii). According to Corollary \ref{C;quasithindecomposiiton}, there exist $c, a_0, a_1,\ldots, a_c\in \mathbb{N}_0$ such that
\begin{align*}\label{Eq:9}
\tag{4.2}B_{ab}=\sum_{h=0}^ce_h(B_{ab})_h,
\end{align*}
where we have $0\neq e_h\in \F$, $(B_{ab})_h\in \mathcal{S}_{a_h}$, and $(B_{ab})_h\notin \sum_{k=-1}^{a_h-1}\mathcal{T}_k$ for any $h\in [c]$. Note that there is $m\in [c]$ such that $a_m>0$. Otherwise, notice that $(B_{ab})_h\in \mathcal{S}_0$ for any $h\in [c]$. As we have already known that $|R_{a'}R_b|=1$, by \eqref{Eq:9}, Lemmas \ref{L;Badpairlemma} (ii), and \ref{L;Complexproductone}, there exists some $n\in \F$ such that  $$B_{ab}=E_a^*B_{ab}E_b^*=nE_a^*JE_b^*\in \langle \mathcal{W}\rangle_\F,$$
which contradicts Lemma \ref{L;Badpairlemma} (iii). So there exists $m\in [c]$ such that $a_m>0$. Assume that $(B_{ab})_m=\prod_{q=0}^{a_m}(E_{i_q}^*A_{j_q}E_{\ell_q}^*)$. Since $a_m>0$, by Corollary \ref{C;quasithindecomposiiton} (i), (ii), (iii), we have $k_{i_0}=k_{j_0}=k_{\ell_0}=k_{j_1}=k_{\ell_1}=2$ and $p_{i_0j_0}^{\ell_0}=p_{\ell_0j_1}^{\ell_1}=|R_{i_0'}R_{\ell_1}|=1$. We thus can get a contradiction by setting $u=i_0$, $v=j_0$, $w=\ell_0$, $y=j_1$, and $z=\ell_1$. So we have $\mathcal{S}=\varnothing$. The other direction is shown. (i) is proved.

For (ii), by Corollary \ref{C;quasithindecomposiiton}, notice that there exist $s, b_0, b_1,\ldots, b_s\in \mathbb{N}_0$ and
\begin{align*}\label{Eq:10}
\tag{4.3}M=\sum_{g=0}^sf_gM_g,
\end{align*}
where we have $0\neq f_g\in \F$, $M_g\in \mathcal{S}_{b_g}$, and $M_g\notin \sum_{k=-1}^{b_g-1}\mathcal{T}_k$ for any $g\in [s]$. For any $i\in [s]$, assume that $M_i=\prod_{f=0}^{b_i}(E_{r_f}^*A_{s_f}E_{t_f}^*)$ as $M_i\in \mathcal{S}_{b_i}$. If $b_i=0$ and $|R_{r_0'}R_{t_0}|=1$, by Lemma \ref{L;Complexproductone}, notice that $O\neq M_i=E_{r_0}^*A_{s_0}E_{t_0}^*=E_{r_0}^*JE_{t_0}^*\in \mathcal{W}\subseteq\B$. If $b_i=0$ and $|R_{r_0'}R_{t_0}|>1$, as $S$ is a quasi-thin scheme, notice that $k_{r_0}=k_{t_0}=|R_{r_0'}R_{t_0}|=2$ and $p_{r_0s_0}^{t_0}=1$ by Lemma \ref{L;Intersectionnumber} (v), (iv), and (ii). So $(r_0, t_0)\in \mathcal{R}$. By Corollary \ref{C;badpairform}, the definition of $B_{r_0t_0}$, and \eqref{Eq:preliminary2}, we have $M_i=B_{r_0t_0}$ or $M_i=E_{r_0}^*JE_{t_0}^*-B_{r_0t_0}$. If $b_i>0$, by Corollary \ref{C;quasithindecomposiiton} (i) and (iii), notice that $(r_0, t_{b_i})\in \mathcal{S}$. By Corollary \ref{C;badpairform}, the definition of $B_{r_0t_{b_i}}$, and \eqref{Eq:preliminary2}, we also have $M_i=B_{r_0t_{b_i}}$ or $M_i=E_{r_0}^*JE_{t_{b_i}}^*-B_{r_0t_{b_i}}$. So $M_i$ can be written as an $\F$-linear combination of the members of $\mathcal{B}$. As $i$ is chosen from $[s]$ arbitrarily, (ii) thus follows by \eqref{Eq:10}.
\end{proof}
By Lemmas \ref{L;Linearlyindependent} and \ref{L;Quasi-thin} (ii), we deduce the following corollary.
\begin{cor}\label{C;basis}
Assume that $S$ is a quasi-thin scheme. Then $\B$ is an $\F$-basis of $\mathcal{T}$.
\end{cor}
Theorem \ref{T;A} is shown by Corollary \ref{C;basis}, Lemmas \ref{L;Badpairlemma} (iv), and \ref{L;Intersectionnumber} (v).

By Theorem \ref{T;A} and Definition \ref{D;badpair}, we get the following corollary.
\begin{cor}\label{C;independent}
The $\F$-dimensions of Terwilliger $\F$-algebras of a quasi-thin scheme are independent of the characteristic $p$ and the choices of vertices of this scheme.
\end{cor}
To deduce two additional corollaries, we need the following lemmas.
\begin{lem}\label{L;T0general}
We have $\mathcal{T}_0$ is a unital $\F$-subalgebra of $M_X(\F)$ only if $\mathcal{S}=\varnothing$.
\end{lem}
\begin{proof}
As $A_i, E_{i}^*\in \mathcal{T}_0$ for any $R_i\in S$, by the definition of $\mathcal{T}$, notice that $\mathcal{T}_0$ is a unital $\F$-subalgebra of $M_X(\F)$ if and only if $\mathcal{T}=\mathcal{T}_0$. By Lemma \ref{L;generalresults} (ii), notice that $\mathcal{T}=\mathcal{T}_0$ if and only if $\mathcal{S}_0$ is an $\F$-basis of $\mathcal{T}$. Assume that $\mathcal{S}\neq \varnothing$. Let $(a,b)\in \mathcal{S}$ and notice that $O\neq B_{ab}\in \mathcal{V}\subseteq \mathcal{T}$ by Lemma \ref{L;Badpairlemma} (iii). We also have $|R_{a'}R_b|=1$ by Definition \ref{D;badpair} (ii). If $\mathcal{S}_0$ is an $\F$-basis of $\mathcal{T}$, by Lemmas \ref{L;Badpairlemma} (ii) and \ref{L;Complexproductone}, there exists some $c\in \F$ such that $B_{ab}=E_a^*B_{ab}E_b^*=cE_a^*JE_b^*\in \langle \mathcal{W}\rangle_\F$, which contradicts Lemma \ref{L;Badpairlemma} (iii). So $\mathcal{S}_0$ is not an $\F$-basis of $\mathcal{T}$. The desired lemma thus follows.
\end{proof}
\begin{lem}\label{L;T0quasi-thin}
Assume that $S$ is a quasi-thin scheme. Then we have $\mathcal{T}_0$ is a unital $\F$-subalgebra of $M_X(\F)$ if and only if there do not exist $R_u, R_v, R_w, R_y, R_z\in S$ such that $k_u=k_v=k_w=k_y=k_z=2$ and $p_{uv}^w=p_{wy}^z=|R_{u'}R_z|=1$.
\end{lem}
\begin{proof}
As $A_i, E_{i}^*\in \mathcal{T}_0$ for any $R_i\in S$, by the definition of $\mathcal{T}$, we have $\mathcal{T}_0$ is a unital $\F$-subalgebra of $M_X(\F)$ if and only if $\mathcal{T}=\mathcal{T}_0$. By Lemma \ref{L;generalresults} (ii), $\mathcal{T}=\mathcal{T}_0$ if and only if $\mathcal{S}_0$ is an $\F$-basis of $\mathcal{T}$. As $S$ is a quasi-thin scheme, by Theorem \ref{T;A}, Lemmas \ref{L;generalresults} (ii), and \ref{L;Nonzerointersection}, $\mathcal{S}_0$ is an $\F$-basis of $\mathcal{T}$ if and only if $\mathcal{S}=\varnothing$. So $\mathcal{T}_0$ is a unital $\F$-subalgebra of $M_X(\F)$ if and only if $\mathcal{S}=\varnothing$. So we can complete the proof by Lemma \ref{L;Quasi-thin} (i).
\end{proof}
\begin{defn}\label{D;triplyregular}\cite[Page 109]{W}
\em The scheme $S$ is called a triply regular scheme if, for all $y\in X$, the $\F$-linear space generated by all triple products of $\mathcal{T}(y)$ is always a unital $\F$-subalgebra of $M_X(\F)$.
\end{defn}
The following corollary is shown by Lemma \ref{L;T0general} and Definition \ref{D;triplyregular}.
\begin{cor}\label{C;triplyregulargeneral}
The scheme $S$ is a triply regular scheme only if $\mathcal{S}=\varnothing$.
\end{cor}
As the fixed vertex $x$ is chosen from $X$ arbitrarily, we get the following corollary by Lemma  \ref{L;T0quasi-thin} and Definition \ref{D;triplyregular}.
\begin{cor}\label{C;triplyregularquasi-thin}
Assume that $S$ is a quasi-thin scheme. Then the scheme $S$ is a triply regular scheme if and only if there do not exist $R_u, R_v, R_w, R_y, R_z\in S$ such that $k_u=k_v=k_w=k_y=k_z=2$ and $p_{uv}^w=p_{wy}^z=|R_{u'}R_z|=1$.
\end{cor}
We conclude this section by some examples of Corollary \ref{C;triplyregularquasi-thin}.
\begin{eg}\label{E;example2}
\em Assume that $S$ is the scheme of order $6$, No. $5$ in \cite{HM}. Notice that $S=\{R_0, R_1, R_2, R_3\}$, where $k_0=k_1=1$ and $k_2=k_3=2$. By Corollary \em \ref{C;triplyregularquasi-thin}, \em $S$ is a triply regular scheme as $|R_2R_2|=|R_2R_3|=|R_3R_2|=|R_3R_3|=2$. Assume that $S$ is the quasi-thin scheme in Example \em \ref{E;example1}. \em As we have $k_4=k_7=k_8=k_9=k_{13}=2$ and $p_{47}^9=p_{9 (13)}^8=|R_{4'}R_8|=1$, by Corollary \em \ref{C;triplyregularquasi-thin}, \em $S$ is not a triply regular scheme.
\end{eg}
\section{Semisimplicity of Terwilliger $\F$-algebras of quasi-thin schemes}
In this section, we determine the semisimple Terwilliger $\F$-algebras of quasi-thin schemes,
which establishes Theorem \ref{T;B}. For our purpose, recall that $\mathcal{J}$ denotes the Jacobson radical of $\mathcal{T}$. Moreover, we also recall Notation \ref{N;notation} and the definition of $E_{ab}$ in Subsection \ref{Sc:subsection}. We first provide some preliminary results.
\begin{lem}\label{L;p'-Valencedcase}
Assume that $S$ is a $p'$-valenced scheme. Let $M\in \mathcal{J}$.
\begin{enumerate}[(i)]
\item [\em (i)] There does not exist $R_a\in S$ such that $k_a=1$ and $E_a^*M\neq O$.
\item [\em (ii)] There does not exist $R_a\in S$ such that $k_a=1$ and $ME_{a}^*\neq O$.
\end{enumerate}
\end{lem}
\begin{proof}
Since $M\in \mathcal{J}\subseteq\mathcal{T}$, by Lemma \ref{L;Linearsolution} (iii), there exist $c, a_0, a_1,\ldots, a_c\in \mathbb{N}_0$ such that
$$M=\sum_{b=0}^ce_bM_b,$$
where we have $0\neq e_b\in \F$, $M_b\in \mathcal{S}_{a_b}$, and $M_b\notin\sum_{h=-1}^{a_{b}-1}\mathcal{T}_h$ for any $b\in [c]$. Moreover, let $b\in [c]$. If $a_b>0$ and $M_b=\prod_{u=0}^{a_b}(E_{i_u}^*A_{j_u}E_{\ell_u}^*)$, then we have $\min\{ k_{i_u}, k_{j_u}, k_{\ell_u}\}>1$ for any $u\in [a_b]$. For any $R_g\in S$ and $k_g=1$, we thus can deduce that
\begin{align*}\label{Eq:15}
\tag{5.1}E_g^*M=\sum_{h=0}^ce_hE_{g}^*M_h=\sum_{\substack{h=0\\
                  a_h=0\\
                  }}^ce_hE_{g}^*M_h\in \mathcal{T}_0.
\end{align*}

For (i), suppose that there exists $R_a\in S$ such that $k_a=1$ and $E_a^*M\neq O$. We can choose $k\in X$ such that $xR_a=\{k\}$. By \eqref{Eq:15}, note that
\begin{align*}\label{Eq:16}
\tag{5.2} O\neq E_a^*M=\sum_{v=0}^d\sum_{w=0}^d r_{vw}E_{a}^*A_vE_w^*,
\end{align*}
where $r_{vw}\in \F$ for any $v, w\in [d]$. As we have $O\neq E_a^*M=E_a^*MI=\sum_{q=0}^dE_a^*ME_q^*$ and $k_{a'}=k_a=1$ by Lemma \ref{L;Intersectionnumber} (i), by \eqref{Eq:16}, Lemmas \ref{L;Intersectionnumber} (v), and \ref{L;Complexproductone}, there exist $y,z\in [d]$ such that
\begin{align*}\label{Eq:17}
\tag{5.3} O\neq E_a^*ME_z^*=\sum_{s=0}^dr_{sz}E_a^*A_sE_z^*=&r_{yz}E_a^*A_yE_z^*\\
=&r_{yz}(\sum_{s=0}^dE_a^*A_sE_z^*)=r_{yz}E_a^*JE_z^*.
\end{align*}
Since $M\in \mathcal{J}$, by \eqref{Eq:17}, notice that $E_a^*JE_z^*\in \mathcal{J}$. Therefore we have $E_a^*JE_z^*JE_a^*\in \mathcal{J}$ as $JE_a^*\in \mathcal{T}$. We thus have $JE_z^*J=\overline{k_z}J$ and $\overline{k_z}E_a^*JE_a^*=E_a^*JE_z^*JE_a^*\in \mathcal{J}$, which implies that $E_{kk}=E_a^*JE_a^*\in \mathcal{J}$ as $S$ is a $p'$-valenced scheme and \eqref{Eq:preliminary2} holds. We thus can obtain a contradiction as $E_{kk}$ is not a nilpotent matrix. So there does not exist $R_a\in S$ such that $k_a=1$ and $E_a^*M\neq O$. (i) is shown.

For (ii), as $M\in \mathcal{J}\subseteq \mathcal{T}$ and $\mathcal{T}^t=\mathcal{T}$, we have $M^t\in \mathcal{J}$ by Lemma \ref{L;Radical}. If there is $R_a\in S$ such that $k_a=1$ and $ME_a^*\neq O$, then $E_a^*M^t\neq O$, which contradicts (i). So there does not exist $R_a\in S$ such that $k_a=1$ and $ME_a^*\neq O$. (ii) is proved.
\end{proof}
The following lemma can be proved by computation and induction.
\begin{lem}\label{L;Inductionmatrix}
Let $a\in \mathbb{N}$ and $b, c\in X$. Then we have $$(E_{bb}-E_{bc}-E_{cb}+E_{cc})^a=\overline{2}^{a-1}(E_{bb}-E_{bc}-E_{cb}+E_{cc}).$$
\end{lem}
\begin{lem}\label{L;p>2Quasi-thin}
Assume that $p\neq 2$ and $S$ is a quasi-thin scheme. We have $\mathcal{J}=\{O\}$. In particular, $\mathcal{T}$ is semisimple.
\end{lem}
\begin{proof}
Suppose that $\mathcal{J}\neq \{O\}$ and $O\neq M\in \mathcal{J}\subseteq \mathcal{T}$. According to Corollary \ref{C;quasithindecomposiiton}, there exist $c, a_0, a_1,\ldots ,a_c\in \mathbb{N}_0$ such that
\begin{align*}\label{Eq:18}
\tag{5.4} M=\sum_{b=0}^ce_bM_b,
\end{align*}
where we have $0\neq e_b\in \F$, $M_b\in \mathcal{S}_{a_b}$, and $M_b\notin\sum_{h=-1}^{a_{b}-1}\mathcal{T}_h$ for any $b\in [c]$. Moreover, let $b\in [c]$. If $a_b>0$ and $M_b=\prod_{u=0}^{a_b}(E_{i_u}^*A_{j_u}E_{\ell_u}^*)$, then we have $k_{i_u}=k_{j_u}=k_{\ell_u}=2$ and $\ell_v=i_{v+1}$ for any $u\in [a_b]$ and $v\in [a_b-1]$.

As $p\neq2$ and $S$ is a quasi-thin scheme, notice that $S$ is a $p'$-valenced scheme. By Lemma \ref{L;p'-Valencedcase} (i) and (ii), we have
$$O\neq M=IMI=(\sum_{e=0}^dE_e^*)M(\sum_{f=0}^dE_f^*)=\!\!\sum_{\substack{e,f=0\\
                  k_e=k_f=2\\
                  }}^dE_e^*ME_f^*,$$
which implies that there exist $R_g, R_h\in S$ such that $k_g=k_h=2$ and $O\neq E_g^*ME_h^*$. So we can set $xR_g=\{r_1, r_2\}$ and $xR_h=\{s_1, s_2\}$. According to \eqref{Eq:18}, we thus can get  $E_g^*ME_h^*=\sum_{b=0}^ce_bE_g^*M_bE_h^*$, where, for any $y, z\in[2]\setminus\{0\}$, the $r_y$-row sum of every $e_bE_g^*M_bE_{h}^*$ equals its $s_z$-column sum by Lemma \ref{L;Case4}. In particular, the $r_y$-row sum of $E_g^*ME_h^*$ equals its $s_z$-column sum for any $y,z\in [2]\setminus\{0\}$. As all possible nonzero entries of $E_g^*ME_h^*$ are exactly its $(r_1, s_1)$-entry, $(r_1, s_2)$-entry, $(r_2, s_1)$-entry, and  $(r_2, s_2)$-entry by \eqref{Eq:preliminary2}, we thus have
\begin{align*}\label{Eq:19}
\tag{5.5} O\neq E_g^*ME_h^*=m(E_{r_1s_1}+E_{r_2s_2})+n(E_{r_1s_2}+E_{r_2s_1}),
\end{align*}
where $m, n\in \F$. Since $E_0^*JE_g^*=E_{xr_1}+E_{xr_2}\in\mathcal{T}$ and $E_h^*JE_0^*=E_{s_1x}+E_{s_2x}\in \mathcal{T}$ by \eqref{Eq:preliminary2}, we get that $E_0^*JE_g^*ME_h^*JE_0^*=\overline{2}(m+n)E_{xx}\in \mathcal{J}$ as $M\in \mathcal{J}$. In particular, $\overline{2}(m+n)E_{xx}$ is a nilpotent matrix, which implies that $n=-m$ as $p\neq2$. By \eqref{Eq:19}, we thus have $O\neq E_g^*ME_h^*=m(E_{r_1s_1}-E_{r_1s_2}-E_{r_2s_1}+E_{r_2s_2})$. So $m\neq 0$. We have
\begin{align*}\label{Eq:20}
\tag{5.6} E_g^*ME_h^*M^tE_g^*=E_g^*ME_h^*(E_g^*ME_h^*)^t=\overline{2}m^2(E_{r_1r_1}-E_{r_1r_2}-E_{r_2r_1}+E_{r_2r_2}).
\end{align*}
As $M\in \mathcal{J}\subseteq\mathcal{T}$ and $(E_g^*ME_h^*)^t\in\mathcal{T}$, observe that $E_g^*ME_h^*M^tE_g^*\in \mathcal{J}$. However, as $p\neq2$, note that $E_g^*ME_h^*M^tE_g^*$ is not a nilpotent matrix by \eqref{Eq:20} and Lemma \ref{L;Inductionmatrix}. So we have $E_g^*ME_h^*M^tE_g^*\notin\mathcal{J}$, which is a contradiction. We thus have $\mathcal{J}=\{O\}$, which implies that $\mathcal{T}$ is semisimple. The proof is now complete.
\end{proof}
\begin{eg}\label{E;example3}
\em In general case, $\mathcal{T}$ is not always semisimple even if $S$ is a $p'$-valenced scheme. For a
counterexample, assume that $p=2$ and $S$ is the scheme of order $15$, No. $5$ in \cite{HM}. Note
that $S$ is a $2'$-valenced scheme. Assume that $X=[15]\setminus\{0\}$ and $x=1$. By \cite[5.1]{Han1}, $\mathcal{T}$ is not semisimple.
\end{eg}
The proof of the following lemma is clear. Therefore we omit its proof.
\begin{lem}\label{L;Thinschemecase}
Assume that $S$ is a thin scheme.
\begin{enumerate}[(i)]
\item [\em (i)] We have $\mathcal{T}=M_X(\F)$.
\item [\em (ii)] We have $\mathcal{T}$ is a simple $\F$-algebra with $\F$-dimension $|X|^2$.
\end{enumerate}
\end{lem}
To prove Theorem \ref{T;B}, we need the following result discovered by Hanaki.
\begin{thm}\label{T;Maincitation}\cite[Theorem 3.4]{Han1}
The $\F$-algebra $\mathcal{T}$ is semisimple only if the scheme $S$ is a $p'$-valenced scheme.
\end{thm}
Theorem \ref{T;B} is proved by Theorem \ref{T;Maincitation}, Lemmas \ref{L;p>2Quasi-thin}, and \ref{L;Thinschemecase} (ii).

We deduce a corollary of Theorem \ref{T;B} as follows. To deduce this corollary, we need to introduce another notation and an additional lemma.
\begin{nota}\label{N;notation2}
\em For any $R_a, R_b\in S$ and $\mathcal{X}\subseteq \mathcal{T}$, set $E_a^*\mathcal{X}E_b^*=\{E_a^*ME_b^*: M\in \mathcal{X}\}$. Notice that $E_a^*\mathcal{T}E_a^*$ is an $\F$-subalgebra of $\mathcal{T}$ with the identity element $E_a^*$.
\end{nota}
\begin{lem}\label{L;Global-Local}
The following assertions hold.
\begin{enumerate}[(i)]
\item [\em (i)] \cite[Proposition 3.1]{Han1} The $\F$-algebra $\mathcal{T}$ \!is semisimple only if $E_a^*\mathcal{T}E_a^*$ is semisimple for all $R_a\in S$.
\item [\em (ii)] \cite[Lemma 3.3]{Han1} If $R_a\in S$, $\langle E_a^*JE_a^*\rangle_\F$ is a nonzero two-sided ideal of $E_a^*\mathcal{T}E_a^*$.
\end{enumerate}
\end{lem}
We are now ready to state our corollary.
\begin{cor}\label{C;semicorollary}
Assume that $S$ is a quasi-thin scheme. Then the following assertions are equivalent.
\begin{enumerate}[(i)]
\item [\em (i)] The scheme $S$ is a $p'$-valenced scheme.
\item [\em (ii)]The $\F$-algebra $\mathcal{T}$ is semisimple.
\item [\em (iii)] The $\F$-algebra $E_a^*\mathcal{T}E_a^*$ is semisimple for all $R_a\in S$.
\end{enumerate}
\end{cor}
\begin{proof}
By Theorem \ref{T;B}, (i) is equivalent to (ii). By Lemma \ref{L;Global-Local} (i), (ii) implies (iii). If $S$ is not a $p'$-valenced scheme, then there exists $R_a\in S$ such that $p\mid k_a$. Since $JE_a^*J=\overline{k_a}J=O$, by Lemma \ref{L;Global-Local} (ii), $\langle E_a^*JE_a^*\rangle_\F$ is a nonzero two-sided nilpotent ideal of $E_a^*\mathcal{T}E_a^*$, which implies that $E_a^*\mathcal{T}E_a^*$ is not semisimple. So (iii) implies (i).
\end{proof}
\begin{eg}\label{E;example4}
\em In general case, $\mathcal{T}$ is not always semisimple even if Corollary \em \ref{C;semicorollary} (iii) \em holds. For a counterexample, one can refer to \cite[5.1]{Han1}.
\end{eg}
We end this section with a proposition that may have independent interests.
\begin{prop}\label{P;independentinterest}
Assume that $E_a^*\mathcal{J}E_b^*=\{O\}$ for all distinct $R_a, R_b\in S$. Then $\mathcal{T}$ is semisimple if and only if $E_c^*\mathcal{T}E_c^*$ is semisimple for all $R_c\in S$.
\end{prop}
\begin{proof}
One direction is from Lemma \ref{L;Global-Local} (i). For the other direction, assume that $\mathcal{T}$ is not semisimple. There is $M\in \mathcal{J}$ such that $M\neq O$. By the hypothesis, we have
$$O\neq M=IMI=(\sum_{e=0}^dE_{e}^*)M(\sum_{f=0}^dE_{f}^*)=\sum_{e=0}^dE_e^*ME_e^*,$$
which implies that there exists $R_g\in S$ such that $O\neq E_g^*ME_g^*\in E_g^*\mathcal{J}E_g^*$. Note that $E_g^*\mathcal{J}E_g^*$ is a nonzero two-sided nilpotent ideal of $E_g^*\mathcal{T}E_g^*$, which implies that $E_g^*\mathcal{T}E_g^*$ is not semisimple. So $\mathcal{T}$ is semisimple if $E_c^*\mathcal{T}E_c^*$ is semisimple for all $R_c\in S$.
\end{proof}
\section{Jacobson radicals of Terwilliger $\F$-\!algebras of quasi-thin schemes}
In this section, we show Theorem \ref{T;C}. By Theorems \ref{T;B} and \ref{T;C}, we get the Jacobson radicals of Terwilliger $\F$-algebras of quasi-thin schemes. For our purpose, recall that $\mathcal{J}$ is the Jacobson radical of $\mathcal{T}$. Moreover, we recall Notation \ref{N;notation1} and the definition of $E_{ab}$ in Subsection \ref{Sc:subsection}. We list some preliminary results as follows.
\begin{lem}\label{L;Basistranspose}
If $(a,b)\in \mathcal{U}$, then $(b,a)\in \mathcal{U}$. Moreover, $B_{ab}^t=B_{ba}$.
\end{lem}
\begin{proof}
For the first assertion, as $\mathcal{U}=\mathcal{R}\cup \mathcal{S}$, we check $(b,a)\in \mathcal{U}$ by two cases.
\begin{enumerate}[\text{Case} 1:]
\item $(a,b)\in \mathcal{R}$.
\end{enumerate}
In this case, we have $k_a=k_b=|R_{a'}R_b|=2$. So there exist distinct $R_u, R_v\in S$ such that $p_{a'b}^u>0$ and $p_{a'b}^v>0$. Therefore we have $p_{b'a}^{u'}>0$ and $p_{b'a}^{v'}>0$ by Lemma \ref{L;Intersectionnumber} (i). Observe that $R_{u'}\neq R_{v'}$ since $u\neq v$. So we have $k_a=k_b=2$ and $R_{b'}R_a=\{R_{u'}, R_{v'}\}$ by Lemma \ref{L;Intersectionnumber} (v). We thus have $(b,a)\in \mathcal{R}\subseteq\mathcal{U}$.
\begin{enumerate}[\text{Case} 2:]
\item $(a,b)\in \mathcal{S}$.
\end{enumerate}
In this case, by Definition \ref{D;badpair}, there exist $m\in \mathbb{N}$ and
$$ R_{a_0}, R_{b_0}, R_{c_0}, R_{a_1}, R_{b_1}, R_{c_1}, \ldots, R_{a_m}, R_{b_m}, R_{c_m}\in S$$
such that $a_0=a$, $c_m=b$, $k_{a_h}=k_{c_h}=2$ and $p_{a_hb_h}^{c_h}=|R_{a_0'}R_{c_m}|=1$ for all $h\in [m]$, $c_k=a_{k+1}$ for all $k\in [m-1]$. Set $i_h=c_{m-h}$, $j_h=b_{m-h}'$, and $\ell_h=a_{m-h}$ for all $h\in [m]$ and consider the sequence
\begin{align*}\label{Eq:21}
\tag{6.1}R_{i_0}, R_{j_0}, R_{\ell_0},R_{i_1}, R_{j_1}, R_{\ell_1}, \ldots, R_{i_m}, R_{j_m}, R_{\ell_m}.
\end{align*}
Note that $i_0=c_m=b$ and $\ell_m=a_0=a$. Notice that $k_{i_h}=k_{c_{m-h}}=k_{a_{m-h}}=k_{\ell_h}=2$ for all $h\in [m]$. By Lemma \ref{L;Intersectionnumber} (iv), we thus get  $p_{i_hj_h}^{\ell_h}=p_{c_{m-h}b_{m-h}'}^{a_{m-h}}=p_{a_{m-h}b_{m-h}}^{c_{m-h}}=1$ for all $h\in [m]$. Note that $\ell_k=a_{m-k}=c_{m-k-1}=i_{k+1}$ for all $k\in [m-1]$. As we have $|R_{a'}R_b|=|R_{a_0'}R_{c_m}|=1$, we have $R_{a'}R_b=\{R_n\}\subseteq S$. For any $R_i\in R_{b'}R_a$, note that $p_{b'a}^i>0$ and $p_{a'b}^{i'}>0$ by Lemma \ref{L;Intersectionnumber} (i). So we have $i=n'$ as $R_{a'}R_b=\{R_n\}$. So we have $R_{b'}R_a=\{R_{n'}\}$ and $|R_{i_0'}R_{\ell_m}|=|R_{b'}R_a|=1$. We now can deduce that $(b,a)=(i_0,\ell_m)\in \mathcal{S}\subseteq\mathcal{U}$ as Definition \ref{D;badpair} (i) and (ii) hold for \eqref{Eq:21}.

The first assertion is shown by the listed cases. For the second assertion, since $(a,b)\in \mathcal{U}$, notice that $k_a=k_b=2$ by the definition of $\mathcal{U}$. Therefore we can set $xR_a=\{u_1, u_2\}$ and $xR_b=\{v_1, v_2\}$, where $u_1\prec u_2$ and $v_1\prec v_2$. According to the first assertion, note that $(b,a)\in \mathcal{U}$. By the definitions of $B_{ab}$ and $B_{ba}$, we deduce that $$B_{ab}^t=(E_{u_1v_1}+E_{u_2v_2})^t=E_{v_1u_1}+E_{v_2u_2}=B_{ba}.$$
The second assertion is shown. The proof is now complete.
\end{proof}
The following corollary may have independent interests.
\begin{cor}\label{C;transposebasis}
Let $M\in \mathcal{B}$. Then $M^t\in \mathcal{B}$.
\end{cor}
\begin{proof}
Recall that $\B=\mathcal{V}\cup \mathcal{W}$. If $M\in \mathcal{V}$, there is $(a,b)\in \mathcal{U}$ such that $M=B_{ab}$. By Lemma \ref{L;Basistranspose}, we have $(b,a)\in \mathcal{U}$ and $M^t=B_{ab}^t=B_{ba}\in \mathcal{V}\subseteq\B$. If $M\in \mathcal{W}$, there exist $R_i, R_j\in S$ such that $M=E_i^*JE_j^*$. So we have $M^t=E_j^*JE_i^*\in \mathcal{W}\subseteq \B$. The desired corollary thus follows.
\end{proof}
\begin{lem}\label{L;Computation}
Let $\delta_{\alpha\beta}$ be the Kronecker delta of integers $\alpha,\beta$ whose values are in $\F$.
\begin{enumerate}[(i)]
\item [\em (i)] If $R_a, R_b, R_c, R_e\in S$, then $E_a^*JE_b^*E_c^*JE_e^*=\delta_{bc}\overline{k_b}E_a^*JE_e^*$.
\item [\em (ii)] If $R_a, R_b\in S$ and $(c,e)\in \mathcal{U}$, then $E_a^*JE_b^*B_{ce}=\delta_{bc}E_a^*JE_e^*$.
\item [\em (iii)] If $R_a, R_b\in S$ and $(c,e)\in \mathcal{U}$, then $B_{ce}E_a^*JE_b^*=\delta_{ae}E_c^*JE_b^*$.
\end{enumerate}
\end{lem}
\begin{proof}
For (i), if $b\neq c$, then we have $E_a^*JE_b^*E_c^*JE_e^*=O$. If $b=c$, then we have $E_a^*JE_b^*JE_e^*=E_a^*(JE_b^*J)E_e^*=\overline{k_b}E_a^*JE_e^*$, which completes the proof of (i).

For (ii), since $(c,e)\in \mathcal{U}$, note that $k_c=k_e=2$ by the definition of $\mathcal{U}$. Therefore we can set $xR_c=\{r_1, r_2\}$ and $xR_e=\{s_1, s_2\}$, where $r_1\prec r_2$ and $s_1\prec s_2$. If $b\neq c$, then $E_a^*JE_b^*B_{ce}=E_a^*JE_b^*E_c^*B_{ce}E_e^*=O$ by Lemma \ref{L;Badpairlemma} (ii). If $b=c$, then
\begin{align*}
E_a^*JE_c^*B_{ce}=E_a^*JE_c^*E_c^*B_{ce}E_e^*&=E_a^*JE_c^*B_{ce}E_e^*\\
&=E_a^*JB_{ce}\\
&=E_a^*J(E_{r_1s_1}+E_{r_2s_2})=E_a^*J(E_{s_1s_1}+E_{s_2s_2})=E_a^*JE_e^*
\end{align*}
by Lemma \ref{L;Badpairlemma} (ii) and the definitions of $B_{ce}$, $E_e^*$. The proof of (ii) is now complete.

For (iii), as $(c,e)\in \mathcal{U}$, we have $(e,c)\in \mathcal{U}$ by Lemma \ref{L;Basistranspose}. If $a\neq e$, then we have $B_{ce}E_a^*JE_b^*=(E_b^*JE_a^*B_{ec})^t=O$ by Lemma \ref{L;Basistranspose} and (ii). If $a=e$, then we also have $B_{ce}E_a^*JE_b^*=(E_b^*JE_a^*B_{ec})^t=(E_b^*JE_c^*)^t=E_c^*JE_b^*$ by Lemma \ref{L;Basistranspose} and (ii). The proof of (iii) is now complete.
\end{proof}
For further discussion, we need to introduce the following notation.
\begin{nota}\label{N;notation3}
\em Set $\mathcal{J}_0=\langle \mathcal{W}\rangle_\F$ and put $\mathcal{J}_1=\langle\{E_a^*JE_b^*\in \mathcal{W}: \max\{k_a, k_b\}=2\}\rangle_\F$. Note that $\mathcal{J}_1\subseteq\mathcal{J}_0\subseteq\mathcal{T}$.
\end{nota}
By Notation \ref{N;notation3} and Lemma \ref{L;generalresults} (iii), the proof of the following lemma is obvious.
\begin{lem}\label{L;Transpose}
The following assertions hold.
\begin{enumerate}[(i)]
\item [\em (i)] We have $M\in \mathcal{J}_0$ if and only if $M^t\in \mathcal{J}_0$.
\item [\em (ii)] We have $M\in \mathcal{J}_1$ if and only if $M^t\in \mathcal{J}_1$.
\item [\em (iii)] We have $\mathcal{J}_1$ is the zero space if and only if $k_a\neq 2$ for all $R_a\in S$.
\end{enumerate}
\end{lem}
We are now ready to deduce two corollaries of Lemmas \ref{L;Computation} and \ref{L;Transpose}.
\begin{cor}\label{C;ideal0}
Assume that $S$ is a quasi-thin scheme.
\begin{enumerate}[(i)]
\item [\em (i)] We have $MN, NM\in \mathcal{J}_0$ if $M\in \mathcal{T}$ and $N\in \mathcal{J}_0$.
\item [\em (ii)] We have $\mathcal{J}_0$ is a two-sided ideal of $\mathcal{T}$.
\end{enumerate}
\end{cor}
\begin{proof}
For (i), since $S$ is a quasi-thin scheme, $\B$ is an $\F$-basis of $\mathcal{T}$ by Corollary \ref{C;basis}. Therefore there is no loss to assume that $M\in \B$ as $\mathcal{J}_0$ is an $\F$-linear space. Since $\mathcal{J}_0=\langle\mathcal{W}\rangle_\F$, there is also no loss to assume that $N\in \mathcal{W}$. As $\B=\mathcal{V}\cup\mathcal{W}$, by Lemma \ref{L;Computation} (i), (ii), and (iii), we have $MN, NM\in \mathcal{J}_0$. (i) thus follows. As $\mathcal{J}_0=\langle\mathcal{W}\rangle_\F\subseteq\mathcal{T}$, (ii) is shown by (i). The proof is now complete.
\end{proof}
\begin{cor}\label{C;ideal1}
Assume that $p=2$ and $S$ is a quasi-thin scheme.
\begin{enumerate}[(i)]
\item [\em (i)] We have $MN, NM\in \mathcal{J}_1$ if $M\in \mathcal{T}$ and $N\in \mathcal{J}_1$.
\item [\em (ii)] We have $N^3=O$ if $N\in \mathcal{J}_1$.
\item [\em (iii)] We have $\mathcal{J}_1$ is a two-sided nilpotent ideal of $\mathcal{T}$. In particular, $\mathcal{J}_1\subseteq\mathcal{J}$.
\end{enumerate}
\end{cor}
\begin{proof}
If $S$ is a thin scheme, then (i), (ii), and (iii) hold trivially by Lemma \ref{L;Transpose} (iii). We assume further that $S$ is not a thin scheme. So $\mathcal{J}_1\neq \{O\}$ by Lemma \ref{L;Transpose} (iii).

For (i), since $S$ is a quasi-thin scheme, $\B$ is an $\F$-basis of $\mathcal{T}$ by Corollary \ref{C;basis}. So there is no loss to assume that $M\in \B$ as $\mathcal{J}_1$ is an $\F$-linear space. Since $\mathcal{J}_1$ is an $\F$-linear space and $S$ is a quasi-thin scheme, there is also no loss to assume that $N=E_a^*JE_b^*\in \mathcal{W}$ and $\max\{k_a, k_b\}=2$. If $M=B_{ce}\in \mathcal{V}\subseteq\B$, then $k_c=k_e=2$ by the definition of $\mathcal{V}$. So we have $MN=B_{ce}E_a^*JE_b^*\in \mathcal{J}_1$ and
$NM=E_a^*JE_b^*B_{ce}\in\mathcal{J}_1$ by Lemma \ref{L;Computation} (iii) and (ii). If $M=E_i^*JE_j^*\in \mathcal{W}\subseteq\B$, suppose that $MN\notin \mathcal{J}_1$. So $O\neq MN=E_i^*JE_j^*E_a^*JE_b^*=\overline{k_a}E_i^*JE_b^*\notin \mathcal{J}_1$, which implies that $k_i=k_b=1$. Since $\max\{k_a, k_b\}=2=p$, we get $k_a=2$ and $O=\overline{2}E_i^*JE_b^*=MN\neq O$, which is a contradiction. So we have $MN\in \mathcal{J}_1$. As $M^t=E_j^*JE_i^*\in \mathcal{W}$ and $N^t=E_b^*JE_a^*\in \mathcal{J}_1$ by Lemma \ref{L;Transpose} (ii), we can get a similar contradiction if we suppose that $M^tN^t\notin \mathcal{J}_1$. So we have $NM=(M^tN^t)^t\in \mathcal{J}_1$ by Lemma \ref{L;Transpose} (ii). (i) thus follows as $\B=\mathcal{V}\cup \mathcal{W}$.

For (ii), we claim that $E_g^*JE_h^* E_u^*JE_v^*E_y^*JE_z^*=O$ if $E_g^*JE_h^*, E_u^*JE_v^*, E_y^*JE_z^*\in \mathcal{W}$ and $\max\{k_u, k_v\}=2$. Suppose that $E_g^*JE_h^* E_u^*JE_v^*E_y^*JE_z^*\neq O$. So we have $u=h$, $v=y$, and
$E_g^*JE_u^*JE_v^*JE_z^*=E_g^*(JE_u^*J)E_v^*JE_z^*=\overline{k_u}E_g^*(JE_v^*J)E_z^*=\overline{k_uk_v}E_g^*JE_z^*$. As $\max\{k_u, k_v\}=2=p$, we have $O=\overline{k_uk_v}E_g^*JE_z^*=E_g^*JE_h^* E_u^*JE_v^*E_y^*JE_z^*\neq O$, which is a contradiction. The desired claim is shown. (ii) follows by this claim and the definition of $\mathcal{J}_1$. As $\mathcal{J}_1$ is an $\F$-linear space, (iii) is proved by (i) and (ii).
\end{proof}
\begin{nota}\label{N;notation4}
\em Let $\mathcal{T}/\mathcal{I}$ be the quotient $\F$-algebra of $\mathcal{T}$ with respect to the two-sided ideal $\mathcal{I}$. Assume that $p=2$ and $S$ is a quasi-thin scheme. Then $\mathcal{T}/\mathcal{J}_1$ is defined by Corollary \em \ref{C;ideal1} (iii). \em Let $\mathcal{J}_2$ denote the Jacobson radical of $\mathcal{T}/\mathcal{J}_1$. For any $M\in \mathcal{T}$, let $\underline{M}$ be the image of $M$ under the natural $\F$-algebra homomorphism from $\mathcal{T}$ to $\mathcal{T}/\mathcal{J}_1$.
\end{nota}
\begin{lem}\label{L;Radicalsemisimple}
Assume that $p=2$ and $S$ is a quasi-thin scheme. Then $\mathcal{J}_2=\{\underline{O}\}$.
\end{lem}
\begin{proof}
Let $M\in \mathcal{T}$ and $\underline{M}\in \mathcal{J}_2$. As $S$ is a quasi-thin scheme, by Corollary \ref{C;basis},
\begin{align*}\label{Eq:22}
 \tag{6.2}\underline{M}=\sum_{i=0}^d\sum_{j=0}^dc_{ij}\underline{E_i^*JE_j^*}+\sum_{(a,b)\in\mathcal{U}}e_{ab}\underline{B_{ab}}=\sum_{\substack{i,j=0\\
                  k_i=k_j=1\\
                 }}^dc_{ij}\underline{E_i^*JE_j^*}+\sum_{(a,b)\in\mathcal{U}}e_{ab}\underline{B_{ab}},
\end{align*}
where we have $e_{ab}, c_{ij}\in \F$ for any $(a,b)\in \mathcal{U}$ and $R_i, R_j\in S$. If $\mathcal{U}=\varnothing$, the sum $\sum_{(a,b)\in\mathcal{U}}e_{ab}\underline{B_{ab}}$ is regraded as $\underline{O}$. If $e_{uv}\neq 0$ for some $(u,v)\in \mathcal{U}$, then $k_u=k_v=2$ by the definition of $\mathcal{U}$. So we set $xR_u=\{u_1, u_2\}$ and $xR_v=\{v_1, v_2\}$, where $u_1\prec u_2$ and $v_1\prec v_2$. Observe that $(u,u)\in \mathcal{U}$ as $(u,u)\in \mathcal{R}$. By Lemma \ref{L;Basistranspose}, notice that $(v,u)\in \mathcal{U}$. So $B_{uv}B_{vu}=(E_{u_1v_1}+E_{u_2v_2})(E_{v_1u_1}+E_{v_2u_2})=E_{u_1u_1}+E_{u_2u_2}=B_{uu}$ by the definitions of $B_{uv},B_{vu}, B_{uu}$. According to \eqref{Eq:22} and Lemma \ref{L;Badpairlemma} (ii), (iii), we thus have $\underline{O}\neq e_{uv}\underline{B_{uu}}=e_{uv}\underline{B_{uv}B_{vu}}=\underline{E_u^*ME_v^*B_{vu}}\in\mathcal{J}_2$ as $\underline{M}\in\mathcal{J}_2$. This is absurd as $e_{uv}\underline{B_{uu}}$ is not a nilpotent element of $\mathcal{T}/\mathcal{J}_1$. So $e_{ab}=0$ for all $(a,b)\in\mathcal{U}$. If $c_{yz}\neq 0$ for some $R_y, R_z\in S$ and $k_y=k_z=1$, according to \eqref{Eq:22}, Lemmas \ref{L;Badpairlemma} (ii), and \ref{L;generalresults} (iii), $\underline{O}\neq c_{yz}\underline{E_y^*JE_y^*}=c_{yz}\underline{E_y^*(JE_z^*J)E_y^*}=\underline{E_y^*ME_z^*JE_y^*}\in\mathcal{J}_2$
as $\underline{M}\in \mathcal{J}_2$. This is absurd as $c_{yz}\underline{E_y^*JE_y^*}$ is not a nilpotent element of $\mathcal{T}/\mathcal{J}_1$. So $\underline{M}=\underline{O}$. We are done.
\end{proof}
Theorem \ref{T;C} is shown by Corollary \ref{C;ideal1} (iii) and Lemma \ref{L;Radicalsemisimple}.
\section{Structures of Terwilliger $\F$-algebras of quasi-thin schemes}
In this section, we prove Theorems \ref{T;D} and \ref{T;E} by studying the algebraic structures of Terwilliger $\F$-algebras of quasi-thin schemes. For our purpose, recall that $\mathcal{J}$ is the Jacobson radical of $\mathcal{T}$. Moreover, we recall Notations \ref{N;notation1}, \ref{N;notation3}, \ref{N;notation4}, and the definition of $E_{ab}$ in Subsection \ref{Sc:subsection}. We list some preliminary results as follows.
\begin{lem}\label{L;Basistransitive}
If $(a,b), (b,c)\in \mathcal{U}$, then $(a,c)\in \mathcal{U}$.
\end{lem}
\begin{proof}
As $(a,b), (b,c)\in \mathcal{U}$, notice that $k_a=k_b=k_c=2$ by the definition of $\mathcal{U}$. If $|R_{a'}R_c|=2$, we have $(a,c)\in \mathcal{R}\subseteq \mathcal{U}$. We thus assume further that $|R_{a'}R_c|=1$ by Lemma \ref{L;Intersectionnumber} (v). Since $\mathcal{U}=\mathcal{R}\cup \mathcal{S}$, we distinguish three cases to check $(a,c)\in \mathcal{U}$.
\begin{enumerate}[\text{Case} 1:]
\item $(a,b), (b,c)\in \mathcal{R}$.
\end{enumerate}
In this case, we have $k_a=k_b=k_c=|R_{a'}R_b|=|R_{b'}R_c|=2$. According to Lemma \ref{L;Intersectionnumber} (iv) and (ii), we can deduce that there exist $R_g, R_h\in S$ such that $p_{ag}^b=p_{bh}^c=1$. So we have already got that $k_a=k_b=k_c=2$ and $p_{ag}^b=p_{bh}^c=|R_{a'}R_c|=1$, which implies that $(a,c)\in \mathcal{S}\subseteq\mathcal{U}$ as Definition \ref{D;badpair} (i) and (ii) hold.
\begin{enumerate}[\text{Case} 2:]
\item $(a,b)\in\mathcal{S}$ and $(b,c)\in \mathcal{R}$ or $(a,b)\in\mathcal{R}$ and $(b,c)\in \mathcal{S}$.
\end{enumerate}
We first consider the case $(a,b)\in\mathcal{S}$ and $(b,c)\in \mathcal{R}$.
As $(a,b)\in \mathcal{S}$, by Definition \ref{D;badpair}, there exist $m\in \mathbb{N}$ and
$$ R_{i_0}, R_{j_0}, R_{\ell_0}, R_{i_1}, R_{j_1}, R_{\ell_1},\ldots, R_{i_m}, R_{j_m}, R_{\ell_m}\in S$$
such that $i_0=a$, $\ell_m=b$, $k_{i_n}=k_{\ell_n}=2$ and $p_{i_nj_n}^{\ell_n}=|R_{i_0'}R_{\ell_m}|=1$ for all $n\in [m]$, $\ell_q=i_{q+1}$ for all $q\in [m-1]$. As $(b,c)\in \mathcal{R}$, we have $k_b=k_c=|R_{b'}R_c|=2$. By Lemma \ref{L;Intersectionnumber} (iv) and (ii), we can deduce that there exists $R_e\in S$ such that $p_{be}^c=1$. Set $i_{m+1}=b$, $j_{m+1}=e$, and $\ell_{m+1}=c$. We thus get the sequence
\begin{align*}\label{Eq:23}
\tag{7.1} R_{i_0}, R_{j_0}, R_{\ell_0}, R_{i_1}, R_{j_1}, R_{\ell_1},\ldots, R_{i_m}, R_{j_m}, R_{\ell_m}, R_{i_{m+1}}, R_{j_{m+1}}, R_{\ell_{m+1}},
\end{align*}
where we have $i_0=a$, $\ell_{m+1}=c$, $k_{i_{\tilde{n}}}=k_{\ell_{\tilde{n}}}=2$ and $p_{i_{\tilde{n}}j_{\tilde{n}}}^{\ell_{\tilde{n}}}=|R_{i_0'}R_{\ell_{m+1}}|=|R_{a'}R_c|=1$ for all $\tilde{n}\in [m+1]$, $\ell_{\tilde{q}}=i_{\tilde{q}+1}$ for all $\tilde{q}\in [m]$. So $(a,c)=(i_0, \ell_{m+1})\in \mathcal{S}\subseteq\mathcal{U}$ as Definition \ref{D;badpair} (i) and (ii) hold for \eqref{Eq:23}.

One can mimic the presented proof of the case $(a,b)\in\mathcal{S}$ and $(b,c)\in \mathcal{R}$ to provide a similar proof for the case $(a,b)\in\mathcal{R}$ and $(b,c)\in \mathcal{S}$.
\begin{enumerate}[\text{Case} 3:]
\item $(a,b), (b,c)\in \mathcal{S}$.
\end{enumerate}
In this case, by Definition \ref{D;badpair}, there exist $m_1\in \mathbb{N}, 1<m_2\in \mathbb{N}$, and
\begin{align*}
& R_{a_0}, R_{b_0}, R_{c_0}, R_{a_1}, R_{b_1}, R_{c_1},\ldots, R_{a_{m_1}}, R_{b_{m_1}}, R_{c_{m_1}},\\
& R_{a_{m_1+1}}, R_{b_{m_1+1}}, R_{c_{m_1+1}}, R_{a_{m_1+2}}, R_{b_{{m_1}+2}}, R_{c_{m_1+2}},\ldots, R_{a_{m_1+m_2}}, R_{b_{m_1+m_2}}, R_{c_{m_1+m_2}}\in S
\end{align*}
such that $a_0=a$, $c_{m_1}=a_{m_1+1}=b$, $c_{m_1+m_2}=c$, $k_{a_u}=k_{c_u}=2$ and $p_{a_ub_u}^{c_u}=1$ for all $u\in [m_1+m_2]$, $|R_{a_0'}R_{c_{m_1+m_2}}|=|R_{a'}R_c|=1$ and $c_v=a_{v+1}$ for all $v\in [m_1+m_2-1]$. So $(a,c)=(a_0,c_{m_1+m_2})\in \mathcal{S}\subseteq\mathcal{U}$ as Definition \ref{D;badpair} (i) and (ii) hold for the sequence \begin{align*}
& R_{a_0}, R_{b_0}, R_{c_0}, R_{a_1}, R_{b_1}, R_{c_1},\ldots, R_{a_{m_1}}, R_{b_{m_1}}, R_{c_{m_1}},\\
& R_{a_{m_1+1}}, R_{b_{m_1+1}}, R_{c_{m_1+1}}, R_{a_{m_1+2}}, R_{b_{{m_1}+2}}, R_{c_{m_1+2}},\ldots, R_{a_{m_1+m_2}}, R_{b_{m_1+m_2}}, R_{c_{m_1+m_2}}.
\end{align*}

The desired lemma now follows by all listed cases.
\end{proof}
\begin{lem}\label{L;EquivalenceRelation}
Assume that $\mathcal{A}_2=\{a\in [d]: k_a=2\}\neq \varnothing$. Let $\sim$ be a binary relation on $\mathcal{A}_2$, where, for any $b, c\in \mathcal{A}_2$, we have $b\sim c$ if and only if $(b,c)\in \mathcal{U}$. Then $\sim$ is an equivalence relation on $\mathcal{A}_2$.
\end{lem}
\begin{proof}
For any $e\in \mathcal{A}_2$, $k_{e'}=k_e=2$ by Lemma \ref{L;Intersectionnumber} (i). Observe that $(e,e)\in \mathcal{R}\subseteq\mathcal{U}$. Therefore $e\sim e$, which implies that $\sim$ is reflexive. For any $g, h\in \mathcal{A}_2$, if $g\sim h$, then $(g,h), (h,g)\in \mathcal{U}$ by Lemma \ref{L;Basistranspose}. So $h\sim g$, which implies that $\sim$ is symmetric. For any $i, j, k\in \mathcal{A}_2$, if $i\sim j$ and $j\sim k$, then $(i,j), (j,k), (i,k)\in \mathcal{U}$ by Lemma \ref{L;Basistransitive}. So $i\sim k$, which implies that $\sim$ is transitive. The desired lemma thus follows.
\end{proof}
\begin{lem}\label{L;Secondcomputation}
Assume that $(a,b), (c,e)\in \mathcal{U}$.
\begin{enumerate}[(i)]
\item [\em (i)] We have $B_{ab}B_{ce}=O$ if $b\neq c$.
\item [\em (ii)] We have $B_{ab}B_{be}=B_{ae}\in \mathcal{V}$ if $b=c$.
\end{enumerate}
\end{lem}
\begin{proof}
As $b\neq c$, $B_{ab}B_{ce}=E_a^*B_{ab}E_b^*E_c^*B_{ce}E_e^*=O$ by Lemma \ref{L;Badpairlemma} (ii). (i) is proved.

For (ii), as $b=c$, notice that $(a,e)\in \mathcal{U}$ by Lemma \ref{L;Basistransitive}. So we have $B_{ae}\in \mathcal{V}$. Since $(a,b), (b,e)\in \mathcal{U}$, we have $k_a=k_b=k_e=2$ by the definition of $\mathcal{U}$. So we can set $xR_a=\{u_1, u_2\}$, $xR_b=\{v_1, v_2\}$, and $xR_e=\{w_1, w_2\}$, where $u_1\prec u_2$, $v_1\prec v_2$, and $w_1\prec w_2$. So we have $B_{ab}B_{be}=(E_{u_1v_1}+E_{u_2v_2})(E_{v_1w_1}+E_{v_2w_2})=E_{u_1w_1}+E_{u_2w_2}=B_{ae}$ by the definitions of $B_{ab}, B_{be}$, and $B_{ae}$. The proof of (ii) is now complete.
\end{proof}
We are now ready to prove Theorem \ref{T;D}.
\begin{proof}[Proof of Theorem \ref{T;D}]
Let $y\in X$ and recall that $\mathcal{T}(y)$ is the Terwilliger $\F$-algebra of $S$ with respect to $y$. It suffices to check $\mathcal{T}(y)\cong \mathcal{T}$ as $\F$-algebras.

If $(a,b)\in \mathcal{U}$, notice that $k_a=k_b=2$ by the definition of $\mathcal{U}$. Therefore we can set $yR_a=\{u_1, u_2\}$ and $yR_b=\{v_1, v_2\}$, where $u_1\prec u_2$ and $v_1\prec v_2$. Let $B_{ab}(y)$ denote the matrix $E_{u_1v_1}+E_{u_2v_2}$. For any $R_c\in S$, recall that $E_c^*(y)$ is the dual $\F$-idempotent with respect to $y$ and $R_c$. Write $\B(y)$ for the set $\mathcal{V}(y)\cup \mathcal{W}(y)$, where $$\mathcal{V}(y)=\{B_{gh}(y): (g,h)\in \mathcal{U}\}\ \text{and}\ \mathcal{W}(y)=\{E_i^*(y)JE_j^*(y): R_i, R_j\in S\}.$$

As $S$ is a quasi-thin scheme and the fixed vertex $x$ is chosen from $X$ arbitrarily, observe that $\B(y)$ is an $\F$-basis of $\mathcal{T}(y)$ and $|\B(y)|=|\B|$ by Corollaries \ref{C;basis} and \ref{C;independent}. Moreover, let $\delta_{\alpha\beta}$ denote the Kronecker delta of integers $\alpha$, $\beta$ whose values are in $\F$. By Lemmas \ref{L;Computation} (i), (ii), (iii), \ref{L;Secondcomputation} (i), and (ii), the following assertions also hold.
\begin{enumerate}[(i)]
\item If $R_k, R_\ell, R_m, R_n\in S$, then $E_k^*(y)JE_\ell^*(y)E_m^*(y)JE_n^*(y)=\delta_{\ell m}\overline{k_\ell}E_k^*(y)JE_{n}^*(y)$.
\item If $R_k, R_\ell\in S$ and $(m,n)\in \mathcal{U}$, then $E_k^*(y)JE_\ell^*(y)B_{mn}(y)=\delta_{\ell m}E_k^*(y)JE_n^*(y)$.
\item If $R_k, R_\ell\in S$ and $(m,n)\in \mathcal{U}$, then $B_{mn}(y)E_k^*(y)JE_{\ell}^*(y)=\delta_{nk}E_m^*(y)JE_\ell^*(y)$.
\item If $(k,\ell), (m,n)\in \mathcal{U}$ and $\ell\neq m$, then $B_{k\ell}(y)B_{mn}(y)=O$.
\item If $(k,\ell), (m,n)\in \mathcal{U}$ and $\ell=m$, then $B_{k\ell}(y)B_{\ell n}(y)=B_{kn}(y)\in \mathcal{V}(y)$.
\end{enumerate}
Since $\B(y)$ is an $\F$-basis of $\mathcal{T}(y)$ and $|\B(y)|=|\B|$, by Corollary \ref{C;basis}, there exists an $\F$-linear isomorphism $\phi$ from $\mathcal{T}(y)$ to $\mathcal{T}$ such that  $$\phi(B_{gh}(y))=B_{gh}\ \text{and}\ \phi(E_i^*(y)JE_j^*(y))=E_i^*JE_j^*$$ for any $B_{gh}(y)\in \mathcal{V}(y)$ and $E_i^*(y)JE_j^*(y)\in \mathcal{W}(y)$. As $\B(y)$ is an $\F$-basis of $\mathcal{T}(y)$, by (i), (ii), (iii), (iv), (v), Corollary \ref{C;basis}, Lemmas \ref{L;Computation} (i), (ii), (iii), \ref{L;Secondcomputation} (i), (ii), note that $\phi$ is also an $\F$-algebra isomorphism from $\mathcal{T}(y)$ to $\mathcal{T}$. So $\mathcal{T}(y)\cong \mathcal{T}$ as $\F$-algebras.
\end{proof}
Our next goal is to prove Theorem \ref{T;E}. We first introduce some notations.
\begin{nota}\label{N;notation5}
\em For any $n\in \mathbb{N}$, let $\mathcal{A}_n=\{a\in [d]: k_a=n\}$. Assume that $\mathcal{A}_2\neq \varnothing$. Let $\sim$ be the equivalence relation on $\mathcal{A}_2$ defined in Lemma \em \ref{L;EquivalenceRelation}. \em So there exists $\gamma\in \mathbb{N}_0$ such that $\mathcal{C}_0, \mathcal{C}_1,\ldots, \mathcal{C}_\gamma$ are exactly all pairwise distinct equivalence classes of $\mathcal{A}_2$ with respect to $\sim$. For any $b\in [\gamma]$ and $g,h\in \mathcal{C}_b$, $(g,h)\in \mathcal{U}$ by the definition of $\sim$. Set
\begin{align*}\label{Eq:24}
\tag{7.2}C_{gh}=\begin{cases} B_{gh}-\overline{2}^{-1}E_g^*JE_h^*, & \text{if}\ p\neq 2,\\
B_{gh}, & \text{if}\ p=2.
\end{cases}
\end{align*}
Let $\mathcal{D}_b=\{C_{ij}: i,j\in \mathcal{C}_b\}$ and $\mathcal{D}=\bigcup_{k=0}^{\gamma}\mathcal{D}_k$. Put $\mathcal{C}=\mathcal{D}\cup\mathcal{W}$. 
\end{nota}
\begin{nota}\label{N;notation6}
\em For any nonempty finite set $L$, let $M_{L}(\F)$ be the full matrix algebra of $\F$-square matrices whose rows and columns are labeled by the members of $L$. Let $M_{|L|}(\F)$ denote the full matrix algebra of $\F$-square matrices of size $|L|$ and note that $M_{L}(\F)\cong M_{|L|}(\F)$ as $\F$-algebras.
\end{nota}
To prove Theorem \ref{T;E}, we present the following lemmas as preparation.
\begin{lem}\label{L;Finalbasiscomputation}
Assume that $\mathcal{A}_2\neq \varnothing$.
\begin{enumerate}[(i)]
\item [\em (i)] The set $\mathcal{C}$ is an $\F$-linearly independent subset of $\mathcal{T}$.
\item [\em (ii)] If $C_{ab}, C_{ce}\in \mathcal{D}$ and $b\neq c$, then $C_{ab}C_{ce}=O$.
\item [\em (iii)] If $C_{ab}, C_{bc}\in \mathcal{D}$, then $C_{ab}C_{bc}=C_{ac}\in \mathcal{D}$.
\end{enumerate}
\end{lem}
\begin{proof}
For (i), since $\B=\mathcal{V}\cup \mathcal{W}$ and $\mathcal{C}=\mathcal{D}\cup \mathcal{W}$, every member of $\mathcal{C}$ is an $\F$-linear combination of the members of $\B$. So $\mathcal{C}\subseteq\mathcal{T}$ by Lemma \ref{L;Linearlyindependent}. We now assume that
\begin{align*}\label{Eq:25}
\tag{7.3}\sum_{i=0}^d\sum_{j=0}^dc_{ij}E_i^*JE_j^*+\sum_{C_{gh}\in \mathcal{D}}e_{gh}C_{gh}=O,
\end{align*}
where $c_{ij}, e_{gh}\in \F$ for any $R_i, R_j\in S$ and $C_{gh}\in \mathcal{D}$. For any $C_{gh}\in \mathcal{D}$, by \eqref{Eq:25}, \eqref{Eq:24}, and Lemma \ref{L;Badpairlemma} (ii), there exists $k\in \F$ such that $O=E_g^*OE_h^*=kE_g^*JE_h^*+e_{gh}B_{gh}$, which implies that $e_{gh}=0$ by Lemma \ref{L;Badpairlemma} (iii). By Lemma \ref{L;generalresults} (iii), we also have $c_{ij}=0$ for any $R_i, R_j\in S$. Therefore $\mathcal{C}$ is linearly independent over $\F$. (i) is shown.

For (ii), as $b\neq c$, we have $C_{ab}C_{ce}=E_a^*C_{ab}E_b^*E_c^*C_{ce}E_e^*=O$ by \eqref{Eq:24} and Lemma \ref{L;Badpairlemma} (ii). The proof of (ii) is now complete.

For (iii), as $C_{ab}, C_{bc}\in \mathcal{D}$, we have $a\sim b$ and $b\sim c$ by the definitions of $C_{ab}$ and $C_{bc}$. Therefore $a\sim c$ by Lemma \ref{L;EquivalenceRelation}. So there exists $m\in [\gamma]$ such that $C_{ac}\in \mathcal{D}_m\subseteq \mathcal{D}$.

If $p=2$, we have $C_{ab}C_{bc}=B_{ab}B_{bc}=B_{ac}=C_{ac}\in\mathcal{D}$ by \eqref{Eq:24} and Lemma \ref{L;Secondcomputation} (ii). If $p\neq 2$, as $a\sim b$, we have $(a,b)\in \mathcal{U}$ by the definition of $\sim$. So we have $k_b=2$ by the definition of $\mathcal{U}$. By \eqref{Eq:24}, Lemmas \ref{L;Secondcomputation} (ii), and \ref{L;Computation} (i), (ii), (iii), we deduce that
\begin{align*}
C_{ab}C_{bc}&=(B_{ab}-\overline{2}^{-1}E_a^*JE_b^*)(B_{bc}-\overline{2}^{-1}E_b^*JE_c^*)\\
&=B_{ab}B_{bc}-\overline{2}^{-1}B_{ab}E_b^*JE_c^*-\overline{2}^{-1}E_a^*JE_b^*B_{bc}+\overline{4}^{-1}E_a^*(JE_b^*J)E_c^*\\
&=B_{ac}-\overline{2}^{-1}E_a^*JE_c^*-\overline{2}^{-1}E_a^*JE_c^*+\overline{2}^{-1}E_a^*JE_c^*\\
&=B_{ac}-\overline{2}^{-1}E_a^*JE_c^*=C_{ac}\in \mathcal{D}.
\end{align*}
(iii) is now proved by the conclusions of the cases $p=2$ and $p\neq 2$.
\end{proof}
\begin{lem}\label{L;Newsubalgebras1}
Assume that $\mathcal{A}_2\neq\varnothing$ and $a\in [\gamma]$.
\begin{enumerate}[(i)]
\item [\em (i)] We have $\langle\mathcal{D}_a\rangle_\F$ is an $\F$-subalgebra of $\mathcal{T}$ with the identity element $\sum_{b\in \mathcal{C}_a}C_{bb}$.
\item [\em (ii)] We have $\langle\mathcal{D}_a\rangle_\F\cong M_{\mathcal{C}_a}(\F)\cong M_{|\mathcal{C}_a|}(\F)$ as $\F$-algebras.
\end{enumerate}
\end{lem}
\begin{proof}
For (i), we claim that $C_{ce}C_{gh}\in \langle\mathcal{D}_a\rangle_\F$ if $C_{ce}, C_{gh}\in \mathcal{D}_a$. If $e\neq g$, then we can obtain that $C_{ce}C_{gh}=O\in \langle\mathcal{D}_a\rangle_\F$ by Lemma \ref{L;Finalbasiscomputation} (ii). If $e=g$, as $C_{ce}, C_{eh}\in \mathcal{D}_a$, note that $c,e,h\in \mathcal{C}_a$. So we have $C_{ch}\in \mathcal{D}_a$. We thus have $C_{ce}C_{eh}=C_{ch}\in \mathcal{D}_a$ by Lemma \ref{L;Finalbasiscomputation} (iii). The claim is shown. Since $\langle \mathcal{D}_a\rangle_\F$ is an $\F$-linear space, by the proved claim, notice that $MN\in \langle\mathcal{D}_a\rangle_\F$ if $M, N\in \langle\mathcal{D}_a\rangle_\F$. So it is enough to check that we have $M(\sum_{b\in \mathcal{C}_a}C_{bb})=M=(\sum_{b\in \mathcal{C}_a}C_{bb})M$ if $M\in \langle \mathcal{D}_a\rangle_\F$. As $\langle \mathcal{D}_a\rangle_\F$ is an $\F$-linear space, there is no loss to assume that $M=C_{ij}\in \mathcal{D}_a$. By Lemma \ref{L;Finalbasiscomputation} (ii) and (iii), we have
$C_{ij}(\sum_{b\in \mathcal{C}_a}C_{bb})=C_{ij}=(\sum_{b\in \mathcal{C}_a}C_{bb})C_{ij}$. The proof of (i) is complete.

For (ii), it suffices to show that $\langle \mathcal{D}_a\rangle_\F\cong M_{\mathcal{C}_a}(\F)$ as $\F$-algebras. Notice that $\mathcal{D}_a$ is an $\F$-basis of $\langle \mathcal{D}_a\rangle_\F$ by Lemma \ref{L;Finalbasiscomputation} (i). Let $F_{mn}$ denote the $(0,1)$-matrix of $M_{\mathcal{C}_a}(\F)$ whose
unique nonzero entry is the $(m,n)$-entry. So $\{F_{rs}: r,s\in \mathcal{C}_a\}$ is an $\F$-basis of $M_{\mathcal{C}_a}(\F)$. Let $\phi$ be the $\F$-linear isomorphism from $\langle \mathcal{D}_a\rangle_\F$ to $M_{\mathcal{C}_a}(\F)$ that sends every $C_{uv}$ to $F_{uv}$. By Lemma \ref{L;Finalbasiscomputation} (ii) and (iii), it is not very difficult to notice that $\phi$ is an $\F$-algebra isomorphism from $\langle \mathcal{D}_a\rangle_\F$ to $M_{\mathcal{C}_a}(\F)$. The proof of (ii) is complete.
\end{proof}
The following lemmas focus on the quasi-thin schemes.
\begin{lem}\label{L;Newsubalgebras2}
If $p\neq 2$ and $S$ is a quasi-thin scheme, for any $R_a, R_b \in S$, define
\begin{align*}\label{Eq:26}
\tag{7.4} D_{ab}=\begin{cases} \overline{2}^{-1}E_a^*JE_b^*, & \text{if}\ k_a=k_b=2,\\
E_a^*JE_b^*, &\text{otherwise}.\end{cases}
\end{align*}
\begin{enumerate}[(i)]
\item [\em (i)] We have $\mathcal{J}_0$ is an $\F$-subalgebra of $\mathcal{T}$ with the identity element $\sum_{c=0}^dD_{cc}$.
\item [\em (ii)] We have $\mathcal{J}_0$ is a simple $\F$-subalgebra of $\mathcal{T}$ with $\F$-dimension $(d+1)^2$.
\end{enumerate}
\end{lem}
\begin{proof}
For (i), according to Corollary \ref{C;ideal0} (i), notice that $MN\in \mathcal{J}_0$ if $M,N\in \mathcal{J}_0$. To prove (i), it suffices to check that $M(\sum_{c=0}^dD_{cc})=M=(\sum_{c=0}^dD_{cc})M$ if $M\in \mathcal{J}_0$. As $\mathcal{J}_0=\langle \mathcal{W}\rangle_\F$, there is no loss to assume that $M=E_g^*JE_h^*$, where $R_g, R_h\in S$. As $S$ is a quasi-thin scheme, notice that $\max\{k_g, k_h\}\leq2$. By \eqref{Eq:26} and Lemma \ref{L;Computation} (i), we have $E_g^*JE_h^*(\sum_{c=0}^dD_{cc})=E_g^*JE_h^*=(\sum_{c=0}^dD_{cc})E_g^*JE_h^*$. (i) thus follows.

For (ii), by (i) and Lemma \ref{L;generalresults} (iii), it suffices to check that there does not exist a two-sided ideal $\mathcal{I}$ of $\mathcal{J}_0$ such that $\{O\}\neq \mathcal{I}\neq \mathcal{J}_0$. Suppose that there is a two-sided ideal $\mathcal{I}$ of $\mathcal{J}_0$ such that $\{O\}\neq \mathcal{I}\neq \mathcal{J}_0$. Pick $O\neq w\in \mathcal{I}$. As $\mathcal{J}_0=\langle\mathcal{W}\rangle_\F$, we have
\begin{align*}\label{Eq:27}
\tag{7.5}O\neq w=\sum_{i=0}^d\sum_{j=0}^dc_{ij}E_i^*JE_j^*,
\end{align*}
where $c_{ij}\in \F$ for any $R_i,R_j\in S$. By \eqref{Eq:27}, there must exist some $R_u,R_v\in S$ such that $c_{uv}\neq 0$. For any $R_y,R_z\in S$, since $w\in \mathcal{I}$ and $E_y^*JE_u^*, E_v^*JE_z^*\in \mathcal{J}_0$, we have $E_y^*JE_u^*wE_v^*JE_z^*\!=\!c_{uv}E_y^*(JE_u^*J)E_v^*JE_z^*=\!c_{uv}\overline{k_u}E_y^*(JE_v^*J)E_z^*=c_{uv}\overline{k_uk_v}E_y^*JE_z^*\in\mathcal{J}_0,$
which implies that $E_y^*JE_z^*\in \mathcal{J}_0$ as $c_{uv}\neq 0$, $p\neq 2$, and $\max\{k_u,k_v\}\leq 2$. So $\mathcal{I}=\mathcal{J}_0$ as $\mathcal{J}_0=\langle\mathcal{W}\rangle_\F$. This is a contradiction as we have already assumed that $\mathcal{I}\neq \mathcal{J}_0$. (ii) thus follows.
\end{proof}
\begin{nota}\label{N;notation7}
\em Assume that $p=2$ and $S$ is a quasi-thin scheme. By Theorem \ref{T;C} and Notation \em \ref{N;notation4}, \em let $M\in \mathcal{T}$ and recall that $\underline{M}$ is the image of $M$ under the natural $\F$-algebra homomorphism from $\mathcal{T}$ to $\mathcal{T}/\mathcal{J}$. Let $\underline{N}=\{\underline{L}: L\in N\}$ for any $N\subseteq \mathcal{T}$. Notice that $\underline{N}$ is an $\F$-linear subspace of $\mathcal{T}/\mathcal{J}$ if $N$ is an $\F$-linear subspace of $\mathcal{T}$.
\end{nota}
\begin{lem}\label{L;Newsubalgebra3}
Assume that $p=2$ and $S$ is a quasi-thin scheme.
\begin{enumerate}[(i)]
\item [\em (i)] We have $\underline{\mathcal{J}_0}$ is a two-sided ideal of $\mathcal{T}/\mathcal{J}$.
\item [\em (ii)] We have $\underline{\mathcal{J}_0}$ is an $\F$-subalgebra of $\mathcal{T}/\mathcal{J}$ with the identity element $\sum_{a\in\mathcal{A}_1}\underline{E_a^*JE_a^*}$.
\item [\em (iii)] We have $\underline{\mathcal{J}_0}\cong M_{\mathcal{A}_1}(\F)\cong M_{|\mathcal{A}_1|}(\F)$ as $\F$-algebras.
\end{enumerate}
\end{lem}
\begin{proof}
For (i), as $\underline{\mathcal{J}_0}$ is the image of $\mathcal{J}_0$ under the natural $\F$-algebra homomorphism from $\mathcal{T}$ to $\mathcal{T}/\mathcal{J}$, (i) thus follows by Corollary \ref{C;ideal0} (ii).

For (ii), by Corollary \ref{C;ideal0} (i), we have $\underline{MN}\in\underline{\mathcal{J}_0}$ if $M, N\in \mathcal{J}_0$. To prove (ii), it is now enough to check that $\underline{M}(\sum_{a\in\mathcal{A}_1}\underline{E_a^*JE_a^*})=\underline{M}=(\sum_{a\in\mathcal{A}_1}\underline{E_a^*JE_a^*})\underline{M}$ if $M\in \mathcal{J}_0$. Note that $\underline{\mathcal{J}_0}=\langle\underline{\mathcal{W}}\rangle_\F$. So there is no loss to assume that $\underline{M}=\underline{E_b^*JE_c^*}\in \underline{\mathcal{W}}$. We may also assume further that $\underline{E_b^*JE_c^*}\neq \underline{O}$. Since $p=2$ and $S$ is a quasi-thin scheme, the inequality $\underline{E_b^*JE_c^*}\neq \underline{O}$ implies that $k_b=k_c=1$ by Theorem \ref{T;C}. By Lemma \ref{L;Computation} (i), $(\underline{E_b^*JE_c^*})(\sum_{a\in\mathcal{A}_1}\underline{E_a^*JE_a^*})=\underline{E_b^*JE_c^*}=(\sum_{a\in\mathcal{A}_1}\underline{E_a^*JE_a^*})(\underline{E_b^*JE_c^*})$.
(ii) thus follows.

For (iii), it suffices to prove that $\underline{\mathcal{J}_0}\cong M_{\mathcal{A}_1}(\F)$ as $\F$-algebras. As $p=2$ and $S$ is a quasi-thin scheme, by Lemma \ref{L;generalresults} (iii) and Theorem \ref{T;C}, notice that $\underline{\mathcal{J}_0}$ has an $\F$-basis $\{\underline{E_e^*JE_h^*}: e, h\in \mathcal{A}_1\}$. Let $G_{ij}$ be the $(0,1)$-matrix of $M_{\mathcal{A}_1}(\F)$ whose unique nonzero entry is the $(i,j)$-entry. So $\{G_{k\ell}: k, \ell\in \mathcal{A}_1\}$ is an $\F$-basis of $M_{\mathcal{A}_1}(\F)$. Let $\phi$ be the $\F$-linear isomorphism from $M_{\mathcal{A}_1}(\F)$ to $\underline{\mathcal{J}_0}$ that sends every $G_{uv}$ to $\underline{E_u^*JE_v^*}$. According to Lemma \ref{L;Computation} (i), it is not very difficult to notice that $\phi$ is an $\F$-algebra isomorphism from $M_{\mathcal{A}_1}(\F)$ to $\underline{\mathcal{J}_0}$. The proof of (iii) is now complete.
\end{proof}
\begin{lem}\label{L;Newbasis}
If $S$ is a quasi-thin scheme and $\mathcal{A}_2\neq \varnothing$, then $\mathcal{C}$ is an $\F$-basis of $\mathcal{T}$.
\end{lem}
\begin{proof}
By Lemmas \ref{L;Finalbasiscomputation} (i) and \ref{L;Quasi-thin} (ii), it is enough to check that $\B\subseteq\langle \mathcal{C}\rangle_\F$. Since $\B=\mathcal{V}\cup\mathcal{W}$ and $\mathcal{C}=\mathcal{D}\cup \mathcal{W}$, we only check that $\mathcal{V}\subseteq\langle \mathcal{C}\rangle_\F$. For any $B_{ab}\in \mathcal{V}$, notice that $(a,b)\in \mathcal{U}$ and $a\sim b$. So $C_{ab}$ is defined. Therefore we have $B_{ab}\in \langle \mathcal{C}\rangle_\F$ by \eqref{Eq:24}. The desired lemma follows.
\end{proof}
\begin{lem}\label{L;Newideal1}
Assume that $p\neq 2$, $\mathcal{A}_2\neq \varnothing$, $a\in [\gamma]$, and $S$ is a quasi-thin scheme.
\begin{enumerate}[(i)]
\item [\em (i)] We have $MN, NM\in \langle \mathcal{D}_a\rangle_\F$ if $M\in \mathcal{T}$ and $N\in \langle \mathcal{D}_a\rangle_\F$.
\item [\em (ii)] We have $\langle \mathcal{D}_a\rangle_\F$ is a minimal two-sided ideal of $\mathcal{T}$.
\end{enumerate}
\end{lem}
\begin{proof}
For (i), by the hypotheses and Lemma \ref{L;Newbasis}, notice that $\mathcal{C}$ is an $\F$-basis of $\mathcal{T}$. So there is no loss to assume that $M\in \mathcal{C}$ as $\langle\mathcal{D}_a\rangle_\F$ is an $\F$-linear space. Moreover, there is also no loss to assume further that $N=C_{bc}\in \mathcal{D}_a$. We thus have $(b,c)\in \mathcal{U}$. So $k_b=k_c=2$ by the definition of $\mathcal{U}$. If $M=E_e^*JE_h^*\in \mathcal{W}$, as $p\neq 2=k_b=k_c$,
$$MN=E_e^*JE_h^*(B_{bc}-\overline{2}^{-1}E_b^*JE_c^*)=(B_{bc}-\overline{2}^{-1}E_b^*JE_c^*)E_e^*JE_h^*=NM=O\in \langle\mathcal{D}_a\rangle_\F$$
by \eqref{Eq:24} and Lemma \ref{L;Computation} (i), (ii), (iii). If $M=C_{ij}\in \mathcal{D}$, then $i,j\in \mathcal{C}_k$, where $k\in [\gamma]$. If $k=a$, then $M\in \mathcal{D}_a$. So $MN, NM\in \langle\mathcal{D}_a\rangle_\F$ by Lemma \ref{L;Newsubalgebras1} (i). If $k\neq a$, notice that $\{i,j\}\cap\{b,c\}\subseteq\mathcal{C}_k\cap\mathcal{C}_a=\varnothing$, which implies that $MN=NM=O\in \langle\mathcal{D}_a\rangle_\F$ by Lemma \ref{L;Finalbasiscomputation} (ii). (i) follows as $\mathcal{C}=\mathcal{D}\cup\mathcal{W}$. (ii) follows by (i) and Lemma \ref{L;Newsubalgebras1} (ii).
\end{proof}
\begin{lem}\label{L;Newideal2}
Assume that $p=2$, $\mathcal{A}_2\neq \varnothing$, $a\in [\gamma]$, and $S$ is a quasi-thin scheme.
\begin{enumerate}[(i)]
\item [\em (i)] We have $\underline{MN}, \underline{NM}\in \langle\underline{\mathcal{D}_a}\rangle_\F$ if $M\in \mathcal{T}$ and $N\in \langle\mathcal{D}_a\rangle_\F$.
\item [\em (ii)] We have $\langle\underline{\mathcal{D}_a}\rangle_\F$ is an $\F$-subalgebra of $\mathcal{T}/\mathcal{J}$ with the identity element $\sum_{b\in \mathcal{C}_a}\underline{C_{bb}}$.
\item [\em (iii)] We have $\langle\underline{\mathcal{D}_a}\rangle_\F\cong \langle\mathcal{D}_a\rangle_\F\cong M_{\mathcal{C}_a}(\F)\cong M_{|\mathcal{C}_a|}(\F)$ as $\F$-algebras.
\item [\em (iv)] We have $\langle\underline{\mathcal{D}_a}\rangle_\F$ is a minimal two-sided ideal of $\mathcal{T}/\mathcal{J}$.
\end{enumerate}
\end{lem}
\begin{proof}
For (i), by the hypotheses and Lemma \ref{L;Newbasis}, notice that $\mathcal{C}$ is an $\F$-basis of $\mathcal{T}$. So there is no loss to assume that $M\in\mathcal{C}$ as $\langle\underline{\mathcal{D}_a}\rangle_\F$ is an $\F$-linear space. Moreover, there is also no loss to assume further that $N=C_{ce}\in \mathcal{D}_a$. We thus have $(c,e)\in \mathcal{U}$. So $k_c=k_e=2$ by the definition of $\mathcal{U}$. If $M=E_g^*JE_h^*\in \mathcal{W}$, as $p=k_c=k_e=2$,
$$\underline{MN}=\underline{E_g^*JE_h^*B_{ce}}=\underline{B_{ce}E_g^*JE_h^*}=\underline{NM}=\underline{O}\in
\langle\underline{\mathcal{D}_a}\rangle_\F$$
by \eqref{Eq:24}, Lemma \ref{L;Computation} (ii), (iii), and Theorem \ref{T;C}. If $M=C_{ij}\in \mathcal{D}$, we have $i, j\in \mathcal{C}_k$, where $k\in [\gamma]$. If $k=a$, then $M\in \mathcal{D}_a$. So $\underline{MN}, \underline{NM}\in \langle\underline{\mathcal{D}_a}\rangle_\F$ by Lemma \ref{L;Newsubalgebras1} (i). If $k\neq a$, then $\{i,j\}\cap\{c,e\}\subseteq\mathcal{C}_k\cap\mathcal{C}_a=\varnothing$, which means that $\underline{MN}=\underline{NM}=\underline{O}\in \langle\underline{\mathcal{D}_a}\rangle_\F$ by Lemma \ref{L;Finalbasiscomputation} (ii). (i) follows as $\mathcal{C}=\mathcal{D}\cup\mathcal{W}$.

For (ii), by (i), we only need to check that $\underline{L}(\sum_{b\in \mathcal{C}_a}\underline{C_{bb}})=\underline{L}=(\sum_{b\in \mathcal{C}_a}\underline{C_{bb}})\underline{L}$ if $L\in \langle\mathcal{D}_a\rangle_\F$. As $\langle\mathcal{D}_a\rangle_\F$ is an $\F$-linear space, we may assume further that $L=C_{\ell m}\in \mathcal{D}_a$. By Lemma \ref{L;Finalbasiscomputation} (ii) and (iii), we have $\underline{C_{\ell m}}(\sum_{b\in \mathcal{C}_a}\underline{C_{bb}})=\underline{C_{\ell m}}=(\sum_{b\in \mathcal{C}_a}\underline{C_{bb}})\underline{C_{\ell m}}$. The proof of (ii) is now complete.

For (iii), by Lemma \ref{L;Newsubalgebras1} (ii), it suffices to prove that $\langle\underline{\mathcal{D}_a}\rangle_\F\cong \langle\mathcal{D}_a\rangle_\F$ as $\F$-algebras. By the definition of $\underline{\mathcal{D}_a}$, let $\phi$ denote the surjective map from $\mathcal{D}_a$ to $\underline{\mathcal{D}_a}$ that sends every $C_{nq}$ to $\underline{C_{nq}}$. For any distinct $C_{rs}, C_{uv}\in \mathcal{D}_a$, observe that the equality $\underline{C_{rs}}=\underline{C_{uv}}$ implies that $C_{rs}-C_{uv}\in \langle \mathcal{W}\rangle_\F$ by Theorem \ref{T;C}. This yields a contradiction by the fact $\mathcal{C}=\mathcal{D}\cup\mathcal{W}$ and Lemma \ref{L;Finalbasiscomputation} (i). Therefore $\phi$ is a bijective map from $\mathcal{D}_a$ to $\underline{\mathcal{D}_a}$.
Since $\mathcal{C}=\mathcal{D}\cup\mathcal{W}$, by Theorem \ref{T;C} and Lemma \ref{L;Finalbasiscomputation} (i), notice that $\mathcal{D}_a$ is an $\F$-basis of $\langle\mathcal{D}_a\rangle_\F$ and $\underline{\mathcal{D}_a}$ is an $\F$-basis of $\langle\underline{\mathcal{D}_a}\rangle_\F$. So we can define $\psi$ to be the $\F$-linear isomorphism from $\langle \mathcal{D}_a\rangle$ to $\langle\underline{\mathcal{D}_a}\rangle$ satisfying $\psi(C_{nq})=\underline{C_{nq}}$ for all $C_{nq}\in \mathcal{D}_a$. By Lemma \ref{L;Finalbasiscomputation} (ii) and (iii), it is not very difficult to see that $\psi$ is an $\F$-algebra isomorphism from $\langle\mathcal{D}_a\rangle_\F$ to $\langle\underline{\mathcal{D}_a}\rangle_\F$. The proof of (iii) is now complete.

(iv) is proved by (i) and (iii).
\end{proof}
We are now ready to deduce the following corollaries.
\begin{cor}\label{C;maintheorem1}
If we have $p\neq 2$, $\mathcal{A}_2\neq \varnothing$, and $S$ is a quasi-thin scheme, then there exist minimal two-sided ideals $\mathcal{I}_{-1}, \mathcal{I}_0,\ldots, \mathcal{I}_{\gamma}$ of $\mathcal{T}$ such that the following assertions hold.
\begin{enumerate}[(i)]
\item [\em (i)] For any $a\in [\gamma]\cup\{-1\}$, $\mathcal{I}_a$ is a unital $\F$-algebra.
\item [\em (ii)] The ideal $\mathcal{I}_{-1}$ is a simple $\F$-algebra with $\F$-dimension $(d+1)^2$.
\item [\em (iii)] As $\F$-algebras, $\mathcal{T}=\mathcal{I}_{-1}\oplus \bigoplus_{b=0}^{\gamma}\mathcal{I}_b\cong \mathcal{I}_{-1}\oplus \bigoplus_{b=0}^{\gamma}M_{|\mathcal{C}_b|}(\F)$.
\end{enumerate}
\end{cor}
\begin{proof}
By Corollary \ref{C;ideal0} (ii), Lemmas \ref{L;Newsubalgebras2} (ii), and \ref{L;Newideal1} (ii), $\mathcal{J}_0$ and $\langle\mathcal{D}_c\rangle_\F$ are minimal two-sided ideals of $\mathcal{T}$ for any $c\in [\gamma]$. Set $\mathcal{I}_{-1}=\mathcal{J}_0$ and $\mathcal{I}_e=\langle\mathcal{D}_e\rangle_\F$ for any $e\in [\gamma]$. By Lemmas \ref{L;Newsubalgebras2} (i) and \ref{L;Newsubalgebras1} (i), the desired assertion (i) holds. By Lemma \ref{L;Newsubalgebras2} (ii), the desired assertion (ii) holds. Since  $\mathcal{C}=\mathcal{D}\cup \mathcal{W}$, by Lemmas \ref{L;Newbasis}, \ref{L;Newsubalgebras1} (ii), and the assertion (i), the desired assertion (iii) holds. The desired corollary thus follows.
\end{proof}
\begin{cor}\label{C;maintheorem2}
If we have $p=2$, $\mathcal{A}_2\neq \varnothing$, and $S$ is a quasi-thin scheme, then there exist minimal two-sided ideals $\mathcal{K}_{-1},\mathcal{K}_0,\ldots, \mathcal{K}_{\gamma}$ of $\mathcal{T}/\mathcal{J}$ such that the following assertions hold.
\begin{enumerate}[(i)]
\item [\em (i)] For any $a\in [\gamma]\cup\{-1\}$, $\mathcal{K}_a$ is a unital $\F$-algebra.
\item [\em (ii)]As $\F$-algebras, $\mathcal{T}/\mathcal{J}=\mathcal{K}_{-1}\oplus\bigoplus_{b=0}^{\gamma}\mathcal{K}_b\cong M_{|\mathcal{A}_1|}(\F)\oplus \bigoplus_{b=0}^{\gamma}M_{|\mathcal{C}_b|}(\F)$.
\end{enumerate}
\end{cor}
\begin{proof}
By Lemmas \ref{L;Newsubalgebra3} (i), (iii), and \ref{L;Newideal2} (iv), $\underline{\mathcal{J}_0}$ and $\langle\underline{\mathcal{D}_c}\rangle_\F$ are minimal two-sided ideals of $\mathcal{T}/\mathcal{J}$ for any $c\in [\gamma]$. Set $\mathcal{K}_{-1}=\underline{\mathcal{J}_0}$ and $\mathcal{K}_e=\langle\underline{\mathcal{D}_e}\rangle$ for any $e\in [\gamma]$. By Lemmas \ref{L;Newsubalgebra3} (ii) and \ref{L;Newideal2} (ii), the desired assertion (i) holds. Since $\mathcal{C}=\mathcal{D}\cup\mathcal{W}$, by Theorem \ref{T;C} and Lemma \ref{L;Newbasis}, notice that $\mathcal{T}/\mathcal{J}=\mathcal{K}_{-1}\oplus\bigoplus_{b=0}^{\gamma}\mathcal{K}_b$, which implies that the desired assertion (ii) holds by Lemmas \ref{L;Newsubalgebra3} (iii) and \ref{L;Newideal2} (iii). We are done.
\end{proof}
Theorem \ref{T;E} is proved by Lemmas \ref{L;Thinschemecase} (i), \ref{L;EquivalenceRelation}, Corollaries \ref{C;maintheorem1}, and \ref{C;maintheorem2}.

We close the whole paper by a corollary that is proved by Corollary \ref{C;maintheorem1} and the Artin-Wedderburn Theorem.
\begin{cor}\label{C;finalstructurecorollary}
Assume that $\F$ is an algebraically closed field and $p\neq 2$. If $S$ is a quasi-thin scheme and $\mathcal{A}_2\neq \varnothing$, then $\mathcal{T}\cong M_{d+1}(\F)\oplus\bigoplus_{a=0}^\gamma M_{|\mathcal{C}_a|}(\F)$ as $\F$-algebras.
\end{cor}

\end{document}